\documentclass[a4paper]{article}

\usepackage{amsfonts, amsmath, amsthm, empheq, amssymb, color, indentfirst, physics}

\allowdisplaybreaks[4]
\newtheorem{thm}{Theorem}[section]
\newtheorem{pr}{Proposition}[section]

\newtheorem{lem}{Lemma}[section]
\newtheorem{rem}{Remark}[section]

\renewenvironment{abstract}{%
        \small
        \quotation
         \noindent {\bfseries \abstractname } }%
      {\if@twocolumn\else\endquotation\fi}

\def\R{\mathbb R}
\def\N{\mathbb N}

\def\Om{\Omega}

\def\pa{\partial}

\providecommand{\norm}[1]{\left\lVert#1\right\rVert}

\title{\Large\bf
Determination of source and initial values for acoustic equations with a time-fractional attenuation}

\author{\large Xinchi HUANG$^1$, Yavar KIAN$^2$, \'Eric SOCCORSI$^2$, Masahiro YAMAMOTO$^3$}

\date{}

\begin{document}
\maketitle

\renewcommand{\thefootnote}{\fnsymbol{footnote}}
\footnotetext{\hspace*{-5mm} 
\begin{tabular}{@{}r@{}p{16cm}@{}} 
& Manuscript last updated: \today.\\
$^1$ 
& Graduate School of Mathematical Sciences, 
the University of Tokyo,
3-8-1 Komaba, Meguro-ku, Tokyo 153-8914, Japan.\\
& JSPS Postdoctoral Fellowships for research in Japan.
E-mail: 
huangxc@ms.u-tokyo.ac.jp.\\
$^2$
& Aix-Marseille Univ, Universit\'e de Toulon, CNRS, CPT, Marseille, France. 
E-mails:
yavar.kian@univ-amu.fr,
eric.soccorsi@univ-amu.fr.\\
$^3$ 
& Graduate School of Mathematical Sciences, 
the University of Tokyo,
3-8-1 Komaba, Meguro-ku, Tokyo 153-8914, Japan; \\
& Honorary Member of Academy of Romanian Scientists,
 Ilfov, nr. 3, Bucuresti, Romania,\\
&Correspondence member
of Accademia Peloritana dei Pericolanti, Palazzo Universit\`a, 
Piazza S. Pugliatti 1 98122 Messina, Italy;\\
& Peoples’ Friendship University of Russia (RUDN University) 6 Miklukho-Maklaya St, Moscow, 117198, Russian Federation. 
E-mail: myama@ms.u-tokyo.ac.jp.
\end{tabular}}

\begin{center}
\Large Dedicated to the memory of Victor Isakov (1947-2021)
\end{center}

%%%%%%%%%%%%%%%%%%%%%%%%%%%%%%%%%%%%%%%%%%%%%%%%%
%%%%%%%%%%%%%%%%%%%%%%%%%%%%%%%%%%%%%%%%%%%%%%%%%
\begin{abstract}
We consider the inverse problem of determining the initial states or the source term
of a hyperbolic equation damped by some non-local time-fractional derivative. 
This framework is relevant to
medical imaging such as thermoacoustic or photoacoustic tomography.
We prove a stability estimate for each of these two problems, 
with the aid of a Carleman estimate specifically designed for the governing equation.

\vskip 4.5mm

\noindent\begin{tabular}{@{}l@{ }p{11.5cm}} {\bf Keywords } &
time-fractional wave equation, inverse problem, stability, Carleman inequality.
\end{tabular}

\vskip 4.5mm

%\noindent{\bf AMS Subject Classifications } 35R11, {35R30, 35B35}

\end{abstract}

%%%%%%%%%%%%%%%%%%%%%%%%%%%%%%%%%%%%%%%%%%%%%%%%%
%%%%%%%%%%%%%%%%%%%%%%%%%%%%%%%%%%%%%%%%%%%%%%%%%
\section{Introduction and main results}
\label{sec-intro}
\subsection{Settings and governing equation}

Let $\Omega$ be a bounded domain in $\Bbb{R}^n$ with smooth boundary 
(e.g., of $\mathcal C^2$-class). 
Put $Q:= \Omega\times (0,T)$ and $\Sigma := \partial\Omega\times (0,T)$, 
where $T>0$ is arbitrarily fixed.
In what follows $\mathcal A$ is the differential operator 
$$
\mathcal A u = -\frac{1}{\rho(x)}\sum_{i,j=1}^n \partial_{i}\left(a_{ij}(x)\partial_{j}u\right) 
+ \sum_{j=1}^n b_j(x)\partial_j u + c(x) u,  \quad x\in \overline{\Omega}
$$
with real-valued coefficients $a_{ij}=a_{ji} \in \mathcal C^2(\overline{\Omega})$ for $i,j=1,\ldots,n$, 
$b_j \in L^\infty(\Omega)$ for 
$j=1,\ldots,n$, $c \in L^\infty(\Omega)$ and $\rho \in \mathcal{C}^2( \overline{\Omega})$.
We assume that there exists $a_0>0$ such that 
\begin{equation}
\label{ell}
\sum_{i,j=1}^n a_{ij}(x)\xi_i\xi_j \ge a_0 |\xi|^2, 
\quad x\in \overline{\Omega},\ \xi=(\xi_1,\ldots,\xi_n)\in \R^n
\end{equation}
and we suppose that $\rho_0\le \rho \le \rho_1$ in $\Omega$ for some positive constants 
$\rho_0$ and $\rho_1$.

Let $q \in W^{1,\infty}(\Omega)$, where
$W^{k,p}(X)$, $k\in \mathbb{N}$, $p\in \mathbb{N}\cup \{\infty\}$ denotes
the usual $k$-th order Sobolev space of $L^p$ functions on $X=(0,T)$ or $X=\Omega$, 
and let $\alpha \in W^{1,\infty}(\Omega)$ satisfy
$$ 0 < \alpha(x) \le \alpha_1,\ x \in \Omega, $$
for some constant $\alpha_1 \in (0,1)$. 
We introduce $\partial_t^\gamma$,
the Caputo fractional derivative of order $\gamma \in (0,1)$ with respect to $t$, as 
$$
\partial_t^{\gamma} u(t) := 
\frac{1}{\Gamma(1-\gamma)} \int_0^t (t-s)^{-\gamma}\partial_s u(s) ds,\ t > 0
$$
and we consider the following initial-boundary value problem 
for the time-fractional wave equation
\begin{equation}
\label{sy:intro}
\left\{
\begin{aligned}
& \partial_t^2 u(x,t) + q(x) \partial_t^{\alpha} u(x,t) + \mathcal{A}(x) u(x,t) = F(x,t), &&\quad (x,t)\in Q,\\
& u(x,t) = 0,                                                                                                      &&\quad (x,t)\in \Sigma,\\
& u(x,0) = u_0(x),\quad \partial_t u(x, 0) = u_1(x),                                             &&\quad x\in \Omega,
\end{aligned}
\right.
\end{equation}
where the fractional derivative $\partial_t^\alpha$ is pointwisely defined by
$$ \partial_t^\alpha u(x,t) := \partial_t^{\alpha(x)}u(x,t)\ \ (x,t)\in Q. $$

In the present paper, assuming that $\alpha(x) \in (0,1)$, 
we examine the stability issue in the inverse problems of 
determining 
either the source term $F$ or the initial conditions $(u_0,u_1)$ 
from a single boundary measurement of the solution to \eqref{sy:intro}.
Our method does not work directly for $\alpha(x) \in (1,2)$ and we leave this topic open for further research.

\subsection{Motivation and state of the art}

The inverse problems under consideration in this paper are inspired by thermoacoustic or 
photoacoustic tomography (TAT or PAT) problems, such as the determination of an initial pressure 
generated by microwave or laser excitations from acoustic boundary measurements. 
These problems appear in several medical imaging applications (see e.g., \cite{XW})
and can be regarded as the first inversion in the multiwave imaging modalities (see e.g., \cite{SU}). 
Tomography problems for biological tissues are commonly modeled by 
acoustic equations with frequency-dependent attenuation terms (see \cite{HP,Sz}) in the form of 
time-fractional derivatives, see \cite{CH} and \cite[Chapter 6]{TTS} where acoustic dissipation obeying arbitrary frequency power law is expressed by a Caputo derivative of order $\alpha(x) \in (1,3)$.
In this paper, the hyperbolic equation \eqref{sy:intro} describes the time-evolution of the pressure 
in biological tissues and 
we investigate the inverse problems from the TAT or PAT inversion in heterogeneous tissues.
%

%\subsection{Known results}

The inverse problems of determining source terms and initial conditions in hyperbolic equations 
have received much attention over the last decade from the mathematical community because of 
their applicability to seismology and imaging.
One of the important approaches for solving these problems is the Bukhgeim-Klibanov method 
\cite{BK} based on so-called Carleman estimates (see also \cite{BY17,I1,Kh,Klibanov1992}). 
As for other works on hyperbolic inverse source problems 
in a bounded domain, we can refer for example to \cite[Chapter 7]{I}, \cite{Ya95,Ya99} and 
especially to \cite{JLY} which is concerned with hyperbolic equations with
time-dependent principal parts.
Moreover in \cite{SU1}, the speed or a space-varying part of the source term of the wave equation are recovered with a single measurement.
The above-cited articles are concerned with hyperbolic inverse source problems in a bounded 
spatial domain. Inverse source problems in an unbounded domain were studied in 
\cite{BHKY,HK,HKLZ} and in \cite{HKZ}, 
where the source term and an obstacle were retrieved.
As for the recovery of initial data in the background of the TAT or PAT, we refer the reader to 
\cite{AK,HKN,SU2}.
 
All the aforementioned papers considered inverse problems for classical partial differential equations, 
i.e., partial differential equations without non-local terms. Inverse source and/or initial data problems 
for time-fractional diffusion equations were
examined in \cite{JLLY,KSXY,KY2,KJ,LLY2}, but to our best knowledge \cite{AP} is the only paper 
available in the mathematical literature dealing with the recovery of the initial values of 
a hyperbolic equation with a non-local term. 
In the present paper, we aim to further generalize the results of \cite{AP} by showing determination of 
both the initial states and the source term of hyperbolic equations with more general non-local 
perturbations. 

\subsection{Main results and outline}

We start with the unique existence and the regularity of the solution to the initial-boundary value 
problem \eqref{sy:intro}, needed for the analysis of the inverse problems carried out in this article.
Here and henceforth, we use the notations $H^k(X) = W^{k,2}(X)$ and 
$L^p(X) = W^{0,p}(X)$ for $k, p\in \mathbb{N}\cup \{\infty\}$
and $X=(0,T)$ or $X=\Omega$, where $W^{k,p}(X)$ denotes the usual $k$-th order Sobolev space of 
$p$-th power integrable functions in $X$. 
The following result provides a solution to \eqref{sy:intro} which is needed by the analysis of 
the inverse problems under investigation in this article.

\begin{thm}
\label{thm:FP}
Let $u_0\in H^1_0(\Omega)$, let $u_1\in L^2(\Omega)$ and let $F\in L^1(0,T;L^2(\Omega))$. 
Then, there exists a unique solution  
$u\in \mathcal C([0,T];H^1_0(\Omega))\cap \mathcal C^1([0,T];L^2(\Omega))$ to \eqref{sy:intro}, 
satisfying 
\begin{equation}
\label{eq:reg1}
\norm{u}_{\mathcal C([0,T];H^1_0(\Omega))} + \norm{u}_{\mathcal C^1([0,T];L^2(\Omega))}
\leq 
C\left(\norm{u_0}_{H_0^1(\Omega)}+\norm{u_1}_{L^2(\Omega)} 
+\norm{F}_{L^1(0,T;L^2(\Omega))}\right)
\end{equation}
for some positive constant $C$ depending only on $\Omega$, $T$, $\rho$, $\alpha$, $a_{ij}$, $b_j$, 
$c$ and $q$.

Moreover, if $u_0\in H^2(\Omega)\cap H^1_0(\Omega)$, $u_1\in H^1_0(\Omega)$ and 
$F\in H^1(0,T;L^2(\Omega))$ then we have $u\in \mathcal C([0,T];H^2(\Omega) \cap H_0^1(\Omega))
\cap\mathcal C^1([0,T];H^1_0(\Omega))\cap \mathcal C^2([0,T];L^2(\Omega))$ and the estimate
\begin{equation}
\label{eq:reg2}
\norm{u}_{\mathcal C([0,T];H^2(\Omega))} + \norm{u}_{\mathcal C^1([0,T];H^1(\Omega))} 
+ \norm{u}_{\mathcal C^2([0,T];L^2(\Omega))}
\leq 
C\left(\norm{u_0}_{H^2(\Omega)}+\norm{u_1}_{H^1(\Omega)} 
+\norm{F}_{H^1(0,T;L^2(\Omega))}\right).
\end{equation}
%%%%%% not necessary
\iffalse
Finally, fix $k\geq 1$ and assume that $\Omega$ is a $\mathcal C^{2k}$ domain, 
$q\in W^{2(k-1),\infty}(\Omega)$ and $u_0=u_1=0$. Let  $F\in H^{2k-1}((0,T)\times\Omega)$ satisfy
\begin{equation*}
%\label{t1b}
\partial_t^\ell F(0,x)=0,\quad x\in\Omega,\ \ell=0,\ldots,2k-2.
\end{equation*}
Then problem \eqref{sy:intro} admits a unique solution 
$$u\in \bigcap_{\ell=0}^{2k}\mathcal C^\ell([0,T];H^{2k-\ell}(\Omega))$$ 
satisfying
\begin{equation*}
%\label{t1d}
\sum_{\ell=0}^{2}\norm{u}_{\mathcal C^{\ell}([0,T];H^{2k-\ell}(\Omega))}
\leq C\norm{F}_{H^{2k-1}((0,T)\times\Omega)}.
\end{equation*}
\fi
%%%%%%
\end{thm}

For self-contained description and the convenience of the reader, 
the proof of Theorem \ref{thm:FP} is given in Section \ref{sec:FP}.

The key achievement of this article is the Lipschitz-stable determination of either the initial states 
$u_0$ and $u_1$ or the source term $F$, appearing in \eqref{sy:intro}, by one local Neumann data. 
This dual identification result is established through two different stability inequalities 
that we will make precise below. 
For the sake of simplicity, we state these estimates in the peculiar framework where
$$
\rho = 1\ \mbox{in}\ \Omega\ \quad \mbox{and}\ \quad 
a_{ij} = \delta_{ij}\ \mbox{in}\ \Omega,\ 1 \leq i,j \leq n. 
$$
That is to say that we consider  
the following initial-boundary value problem associated with $\alpha, q$ in $W^{1,\infty}(\Omega)$ and
$b_j, c \in L^\infty(\Omega)$, $1\le j\le n$,
\begin{equation}
\label{sy1}
\left\{
\begin{aligned}
& \partial_t^2 u(x,t) + q(x)\partial_t^{\alpha} u(x,t) - \Delta u(x,t) + B(x) \!\cdot\! \nabla u(x,t) + c(x)u(x,t) 
= F(x,t),                                                                                                        &&\quad (x,t)\in Q,\\
& u(x,t) = 0,                                                                                          &&\quad (x,t)\in \Sigma,\\
& u(x,0) = u_0(x),\quad \partial_t u(x, 0) = u_1(x),                                             &&\quad x\in \Omega
\end{aligned}
\right.
\end{equation}
instead of \eqref{sy:intro}. Here and below we use the notation $B:=(b_1,\ldots,b_n)^{T}$ and 
we suppose that we have
$$  \sum_{j=1}^n \norm{b_j}_{L^{\infty}(\Omega)} +  \norm{c}_{L^{\infty}(\Omega)} 
+ \norm{\alpha}_{W^{1,\infty}(\Omega)} +  \norm{q}_{W^{1,\infty}(\Omega)} \leq M $$
for some {\it a priori} fixed positive constant $M$.
In this context, our first main result establishes that one local Neumann data stably determines 
either of the two initial states $u_0$ or $u_1$, when the other one and the source term $F$ are known.

\begin{thm}
\label{thm:stab1}
There exist $T_0>0$ and a sub-boundary $\Gamma_0 \subset \partial \Omega$ such that 
for all $T \ge T_0$, all
$(u_0,u_1) \in H_0^1(\Omega) \times L^2(\Omega)$ satisfying either
\begin{equation}
\label{assump}
\mbox{(\romannumeral1)}\ \ u_0 = 0\ \mbox{in}\ \Omega\ \hspace*{.5cm} \mbox{or}\ \hspace*{.5cm} 
\mbox{(\romannumeral2)}\ \ u_1 = 0\ \mbox{in}\ \Omega,
\end{equation}
%$$
%u_0 = 0 \quad\mbox{in}\ \Omega \qquad
%\mbox{or}\qquad u_1 = 0 \quad\mbox{in}\ \Omega.
%$$
and all $F \in L^2(Q)$, we have
$$
\|u_0\|_{H^1(\Omega)} + \|u_1\|_{L^2(\Omega)} 
\leq C\left( \|F\|_{L^2(Q)} + \|\partial_\nu u\|_{L^2(\Gamma_0 \times (0,T))} \right).
$$
Here $C$ is a positive constant depending only on $\Omega$, $T$, $\Gamma_0$ and $M$, 
and $u$ is the $\mathcal C^1([0,T]; L^2(\Omega))\cap \mathcal C([0,T]; H_0^1(\Omega))$-solution to 
the initial-boundary value problem \eqref{sy1}.
\end{thm}

We stress out that the sub-boundary $\Gamma_0$ is not arbitrary here. 
As a matter of fact we will see below that it is required that 
$\Gamma_0$ is taken sufficiently large in order to satisfy the geometric condition 
\eqref{con:Gamma0} for some $x_0 \in \R^n \setminus \overline{\Omega}$.

The second main result of this paper claims stable recovery of the spatial varying factor of 
the source term, by one local boundary observation of the solution.
More precisely, if we assume that $F$ is of the form
\begin{equation}
\label{def-F}
F(x,t)=R(x,t) f(x),\ (x,t) \in Q,
\end{equation}
where 
$R\in W^{1,\infty}(Q)$ is known and satisfies
\begin{equation}
\label{con:R}
|R(x,0)| \ge r_0, \quad x\in \overline{\Omega}
\end{equation}
for some $r_0>0$, then one local boundary observation of the solution to \eqref{sy1} stably determines 
the unknown function $f$
in $L^2(\Omega)$.
\begin{thm}
\label{thm:stab}
There exist $T_0>0$ and a sub-boundary $\Gamma_0$ such that for all $T \ge T_0$, all
$f \in L^2(\Omega)$ and all $R \in W^{1,\infty}(Q)$ satisfying \eqref{con:R}, the following estimate
$$
\|f\|_{L^2(\Om)} \leq C \|\pa_t (\pa_\nu u)\|_{L^2(\Gamma_0 \times (0,T))}
$$
holds for some positive constant $C$ depending only on $\Omega$, $T$, $\Gamma_0$, $M$ and 
$r_0$. 
Here $u$ denotes the 
$\mathcal C^2([0,T]; L^2(\Omega))\cap \mathcal C^1([0,T]; H_0^1(\Omega))$-solution
to the initial-boundary value problem \eqref{sy1} where $u_0 = u_1 = 0$ in $\Omega$ and 
$F$ is defined by \eqref{def-F}. 
\end{thm}

Theorems \ref{thm:stab1} and \ref{thm:stab} are proved in Section \ref{sec:pr-thm:stab}. 
We leave it to the reader to check that these two statements remain valid when 
the master equation \eqref{sy1} is replaced by
$$
\partial_t^2 u(x,t) + \sum_{k=1}^N q_k(x)\partial_t^{\alpha_k} u(x,t) + \mathcal{A}(x) u(x,t) = F(x,t), 
\quad (x,t)\in \Omega\times (0,T)
$$
with suitable multi-order coefficients
$0 < \alpha_k  < 1$ in ${\Omega}$, $k=1,\ldots, N$ and appropriate elliptic operator 
$\mathcal{A}$ (see e.g. \cite{BY17, H19}). 
Moreover, we point out that the case of $t$-dependent coefficients can be handled 
by adapting the strategy developed in \cite{YLY18} for a classical hyperbolic equation, 
to the framework of Theorems \ref{thm:stab1} and \ref{thm:stab}.

%Actually, we put the addition terms $\sum_{k=1}^N q_k \partial_t \alpha_k u$ and $\sum_{i,j=1} \partial_t  a_{ij} \partial_i\partial_j u$ into the source term and

%\subsection{Comments about our results}

To our best knowledge, Theorems \ref{thm:stab1} and \ref{thm:stab} are the only mathematical results 
on the recovery of the source term and the initial conditions of a hyperbolic equation with 
a time-fractional damping term.
There is a slightly similar result in \cite{AP} though, where the initial conditions of a hyperbolic equation 
with some non-local zeroth order term associated with a $\mathcal C^2(\overline{Q})$-kernel, 
are retrieved. 
On the other hand, the fractional attenuation that we consider in the present paper is generated by 
the singular kernel $t \mapsto \frac{t^{-\alpha(x)}}{\Gamma(1-\alpha(x))}$. 
We refer the reader to \cite{DGGS} for the physical relevance of non-local singular perturbations of 
hyperbolic systems. 

The derivation of Theorems \ref{thm:stab1} and \ref{thm:stab} boils down to a specific Carleman 
inequality (and a suitable energy estimate) for \eqref{sy1}, stated in Proposition \ref{pr:ce}. 
More precisely, we prove Theorem \ref{thm:stab1} by applying the Bukhgeim-Klibanov method. 
When the attenuation is switched off ($q =0$ in $\Omega$), this is a well-known approach for solving 
inverse problems associated with \eqref{sy:intro}. 
However, when $q \neq 0$, this is no longer the case since the fractional damping term 
$q \partial_t^\alpha u$ is non-local.

The remaining part of this paper is organized as follows.
In Section \ref{sec:FP}, we discuss the unique existence of the solution to \eqref{sy:intro} and 
establish some suitable energy estimate.
In Section \ref{sec:CE}, we design a Carleman inequality for the master equation \eqref{sy1}. 
The proofs of Theorems \ref{thm:stab1} and \ref{thm:stab}, which boil down to the energy estimate of 
Section \ref{sec:FP} and the Carleman inequality of Section \ref{sec:CE}, 
are given in Section \ref{sec:pr-thm:stab}.

%%%%%%%%%%%%%%%%%%%%%%%%%%%%%%%%%%%%%%%%%%%%%%%%%
%%%%%%%%%%%%%%%%%%%%%%%%%%%%%%%%%%%%%%%%%%%%%%%%%
\section{Analysis of the forward problem}
\label{sec:FP}
In this section we prove Theorem \ref{thm:FP} and we establish an energy estimate for the solution to 
the initial boundary value problem \eqref{sy:intro}, needed for the proof of Theorems \ref{thm:stab1} 
and \ref{thm:stab}.

\subsection{Proof of Theorem \ref{thm:FP}}
We shall derive Theorem \ref{thm:FP} from a fixed-point theorem argument and a classical unique 
existence result for hyperbolic equations. We start by showing existence within the class
$\mathcal C([0,T];H^1_0(\Omega)) \cap \mathcal C^1([0,T];L^2(\Omega))$ of the solution to 
\eqref{sy:intro}, depending continuously on the initial values and the source term.

\subsubsection{Solving the direct problem in $\mathcal C([0,T];H^1_0(\Omega)) \cap \mathcal C^1([0,T];L^2(\Omega))$}
\label{sec-dp1}
For $(u_0,u_1) \in \mathcal H:=H^1_0(\Omega)\times L^2(\Omega)$, 
we consider the $\mathcal C([0,T];H^1_0(\Omega)) \cap \mathcal C^1([0,T];L^2(\Omega))$-solution 
$u$ to \eqref{sy:intro} with $q=0$ and $F=0$, given by \cite[Chap. 1, Lemma 3.6]{L}. 
Then, for each $t \in [0,T]$, we introduce the operator 
\begin{equation}
\label{def-U0}
U_0(t) : (u_0,u_1) \to (u(t),\partial_t u(t))
\end{equation} 
and recall that $t\mapsto U_0(t)\in\mathcal C([0,T];\mathcal B (\mathcal H))$. 
Here and henceforth, we denote by $\mathcal B(X,Y)$ the set of linear bounded operators from 
the Banach space $X$ to the Banach space $Y$, 
and we write $\mathcal B(X)$ instead of $\mathcal B(X,X)$. 

If $q=0$ and $F \in L^1(0,T;L^2(\Omega))$, then it is well known that \eqref{sy:intro} admits 
a unique solution $u$ such that 
$U_1:=(u,\partial_t u) \in \mathcal C([0,T];\mathcal H)$ reads
\begin{equation}
\label{t1d}
U_1(t)=U_0(t)(u_0,u_1)+\int_0^tU_0(t-s)(0,F(s))ds,\ t \in [0,T],
\end{equation}
and satisfies the estimate
\begin{equation}
\label{es}
\norm{U_1}_{\mathcal C([0,T];\mathcal H)}
\leq C_0 \left(\norm{u_0}_{H^1(\Omega)}+\norm{u_1}_{L^2(\Omega)}
+\norm{F}_{L^1(0,T;L^2(\Omega))}\right),
\end{equation}
where the constant $C_0:=\norm{U_0}_{\mathcal C([0,T];\mathcal B(\mathcal H))}$ depends only on 
$\Omega$, $T$, $\rho$,  $c$, $b_j$ and $a_{ij}$ for $1 \leq i,j\leq n$.

If $q$ is not uniformly zero we put
\begin{equation}
\label{def-Q}
Q(t):=-\frac{t^{-\alpha}}{\Gamma(1-\alpha)} \left(\begin{array}{ll}0&0\\ 0& q \end{array}\right),\ t\in (0,T]
\end{equation}
and infer from Duhamel's principle that 
$u \in \mathcal C([0,T];H^1_0(\Omega)) \cap \mathcal C^1([0,T];L^2(\Omega))$ solves
\eqref{sy:intro} if and only if $U:=(u,\partial_t u)$ is a $\mathcal C([0,T];\mathcal H)$-solution to 
the integral equation 
\begin{equation}
\label{FP-eq2}
U(t)=(\mathcal GU)(t):=U_1(t)+\int_0^t\int_0^sU_0(t-s)Q(s-\tau)U(\tau) d\tau ds,\ t \in [0,T].
\end{equation}
Thus it is enough to show that some iterate $\mathcal G^m$ of $\mathcal G$ is a contraction mapping 
on $\mathcal C([0,T];\mathcal H)$. 
To do that, we start by noticing for all $t \in (0,T]$ that $Q(t) \in \mathcal B(\mathcal H)$ and that
\begin{equation}
\label{FP-eq3}
\norm{Q(t)}_{\mathcal B(\mathcal H)} \leq C_1 \frac{t^{-\alpha_1}}{\Gamma(1-\alpha_1)},
\end{equation}
where $C_1:= \Gamma(1-\alpha_1) \max\{1,T^{\alpha_1}\}  \norm{q}_{L^\infty(\Omega)}$. 
Therefore, setting
$\tilde{\mathcal G} V:=\mathcal G V - U_1$ for all $V\in \mathcal C([0,T];\mathcal H)$, i.e.
$$
(\tilde{\mathcal G} V)(t) := 
\int_0^t \int_0^s U_0(t-s)Q(s-\tau)V(\tau) d\tau ds,\ t \in [0,T],\ V\in \mathcal C([0,T];\mathcal H)
$$
we get that
$\norm{(\tilde{\mathcal G} V)(t)}_{\mathcal H}
\leq 
C \int_0^t\int_0^s \frac{(s-\tau)^{-\alpha_1}}{\Gamma(1-\alpha_1)}\norm{V(\tau)}_\mathcal H d\tau ds$
with $C:=C_0 C_1$. As a consequence we have
\begin{eqnarray}
\norm{(\tilde{\mathcal G} V)(t)}_{\mathcal H}
&\leq & C \int_0^t \int_\tau^t \frac{(s-\tau)^{-\alpha_1}}{\Gamma(1-\alpha_1)} \norm{V(\tau)}_\mathcal H 
ds d\tau \nonumber\\
%&\leq & C \int_0^t \left(\int_\tau^t \frac{(s-\tau)^{-\alpha_1}}{\Gamma(1-\alpha_1)} ds\right) \norm{V(\tau)}_\mathcal H  %d\tau\\
& \leq & C \int_0^t\frac{(t-\tau)^{1-\alpha_1}}{\Gamma(2-\alpha_1)}\norm{V(\tau)}_\mathcal H d\tau 
\label{eq-tg}
\end{eqnarray}
by Fubini's theorem. 
From this and \eqref{es}-\eqref{FP-eq2} we see that $\mathcal G$ maps 
$\mathcal C([0,T]; \mathcal H)$ into itself. 
Next, by iterating \eqref{eq-tg} we get in the same manner as in the proof of 
\cite[Proposition 1]{FK} (see also \cite[Section 2.3.6]{P99}) that 
\begin{equation*}
%\label{l6d}
\norm{(\tilde{\mathcal G}^m V)(t)}_{\mathcal H}
\leq C^m\int_0^t\frac{(t-\tau)^{m(2-\alpha_1)-1}}{\Gamma(m(2-\alpha_1))}
\norm{V(\tau)}_{\mathcal H}d\tau,\  t\in(0,T),\ m \in \N.
\end{equation*}
Thus, for all $V, W \in \mathcal C([0,T];\mathcal H)$, all $m \in \N$ and all $t\in [0,T]$ we have
$$
\norm{(\mathcal G^m V)(t) - (\mathcal G^m W)(t)}_{\mathcal H} 
= \norm{(\tilde{\mathcal G}^m (V-W))(t)}_{\mathcal H}
\leq C^m \int_0^t \frac{(t-\tau)^{m(2-\alpha_1)-1}}{\Gamma(m(2-\alpha_1))} 
\norm{(V-W)(\tau)}_{\mathcal H}d\tau %\quad t\in(0,T)
$$
and consequently
\begin{eqnarray}
\norm{\mathcal G^m V - \mathcal G^m W}_{\mathcal C([0,T];\mathcal H)} 
&\leq & C^m \left(\int_0^T\frac{t^{m(2-\alpha_1)-1}}{\Gamma(m(2-\alpha_1))}dt\right)
\norm{V-W}_{\mathcal C([0,T];\mathcal H)} \nonumber \\
%& = & \frac{C^nT^{n(2-\alpha_1)}}{n(2-\alpha_1)\Gamma(n(2-\alpha_1))}
%\norm{V_1-V_2}_{\mathcal C([0,T];\mathcal H)}\\
&\leq & \frac{C^m T^{m(2-\alpha_1)}}{\Gamma(m(2-\alpha_1)+1)}
\norm{V-W}_{\mathcal C([0,T];\mathcal H)}, \label{contractive}
\end{eqnarray}
by Young's convolution inequality. Taking $m \in \N$ so large that 
$\frac{C^{m}T^{m(2-\alpha_1)}}{\Gamma(m(2-\alpha_1)+1)} < 1$, which is possible since
$\lim_{p\to+\infty}\frac{C^pT^{p(2-\alpha_1)}}{\Gamma(p(2-\alpha_1)+1)} = 0$, 
we obtain that $\mathcal G^m$ is contractive on $\mathcal C([0,T];\mathcal H)$. 
%%%%%%%%%%%%%%%% not necessary
\iffalse
Moreover, conditions \eqref{l6d}--\eqref{l6e} imply
$$
\begin{aligned}
\norm{\mathcal G^n U}_{\mathcal C([0,T];\mathcal H)} &= \norm{\mathcal G_2^n U + \sum_{k=1}^n \mathcal G_2^{k-1} [\mathcal G_1(u_0,u_1,F)]}_{\mathcal C([0,T];\mathcal H)} \\
\ &\leq \frac{C^nT^{n(2-\alpha)}}{\Gamma(n(2-\alpha)+1)}\norm{U}_{\mathcal C([0,T];\mathcal H)}\\
\ &\ \ \ +\left(\sum_{k=1}^{n} \frac{C^kT^{k(2-\alpha)-1)}}{\Gamma(k(2-\alpha))}\right)(\norm{u_0}_{H^1(\Omega)}+\norm{u_1}_{L^2(\Omega)}+\norm{F}_{L^1(0,T;L^2(\Omega))})\\
\ &\leq \frac{C^nT^{n(2-\alpha)}}{\Gamma(n(2-\alpha)+1)}\norm{U}_{\mathcal C([0,T];\mathcal H)}\\
\ &\ \ \ +\left(\sum_{k=1}^{\infty} \frac{C^kT^{k(2-\alpha)-1)}}{\Gamma(k(2-\alpha))}\right)(\norm{u_0}_{H^1(\Omega)}+\norm{u_1}_{L^2(\Omega)}+\norm{F}_{L^1(0,T;L^2(\Omega))}).
\end{aligned}
$$
Therefore, by eventually increasing the size of $n_0$, we deduce that 
\fi
%%%%%%%%%%%%%%%%
Therefore, $\mathcal G$ admits a unique fixed point $U\in \mathcal C([0,T];\mathcal H)$ 
by the Banach fixed-point theorem. 

Further, putting \eqref{es}, \eqref{FP-eq2} and \eqref{eq-tg} together, we get that
\begin{eqnarray*}
\norm{U(t)}_{\mathcal H} 
%& = & \norm{(\mathcal G U)(t)}_{\mathcal H} \\
& \leq & C \int_0^t \frac{(t-\tau)^{1-\alpha_1}}{\Gamma(2-\alpha_1)} \norm{U(\tau)}_{\mathcal H}d\tau 
+ C_0 \left(\norm{u_0}_{H^1(\Omega)}+\norm{u_1}_{L^2(\Omega)}+\norm{F}_{L^1(0,T;L^2(\Omega))}\right) \\
& \leq &
C\frac{\max\{1,T^{1-\alpha_1}\}}{\Gamma(2-\alpha_1)} \int_0^t\norm{U(\tau)}_{\mathcal H}d\tau 
+ C_0 \left(\norm{u_0}_{H^1(\Omega)}+\norm{u_1}_{L^2(\Omega)}+\norm{F}_{L^1(0,T;L^2(\Omega))}\right).
\end{eqnarray*}
This and Gr\"onwall's inequality yield \eqref{eq:reg1}. 
%%%%%%%%%%%%%%%% not necessary
\iffalse
This proves that \eqref{FP-eq2} has a unique solution $U\in \mathcal C([0,T];\mathcal H)$ 
fulfilling \eqref{es}. 
From this result, assuming that $U=(u,v)$ we deduce that $u\in\mathcal C([0,T];H_0^1(\Omega))$,  
$\partial_t u = v\in \mathcal C([0,T];L^2(\Omega))$, and thus \eqref{sy:intro} admits a unique solution 
$u\in \mathcal C([0,T];H^1_0(\Omega))\cap \mathcal C^1([0,T];L^2(\Omega))$ satisfying
$$
\norm{u}_{\mathcal C([0,T];H^1_0(\Omega))}+\norm{u}_{\mathcal C^1([0,T];L^2(\Omega))}
\leq 
C\left(\norm{u_0}_{H^1(\Omega)}+\norm{u_1}_{L^2(\Omega)}+\norm{F}_{L^1(0,T;L^2(\Omega))}\right).
$$ 
\fi
%%%%%%%%%%%%%%%%

We turn now to proving the second claim of Theorem \ref{thm:FP}. 

\subsubsection{Improved regularity}

We use the same notations and we follow the same path as in Section \ref{sec-dp1}, 
the initial values $(u_0,u_1)$ being taken in 
$\mathcal H:=(H^2(\Omega)\cap H^1_0(\Omega))\times H^1_0(\Omega)$ 
and the source term $F$ in $H^1(0,T;L^2(\Omega))$.
Namely, with reference to \cite[Theorem 2.3 in Chapter 4 and Theorem 7.1 in Chapter 5]{LM2}, 
the mapping $U_0 \in \mathcal C([0,T];\mathcal B(\mathcal H))$ is defined by \eqref{def-U0} where 
$u$ denotes the $\mathcal C([0,T];H^2(\Omega)\cap H^1_0(\Omega)) \cap \mathcal 
C^1([0,T],H^1_0(\Omega))$-solution to \eqref{sy:intro} with $q=0$ and $F=0$. 
Then, we define $U_1 \in \mathcal C([0,T];\mathcal H)$ by \eqref{t1d} and we replace the estimate 
\eqref{es} by
\begin{equation}
\label{es2}
\norm{U_1}_{\mathcal C([0,T];\mathcal H)}
\leq C_0 \left(\norm{u_0}_{H^{2}(\Omega)}+\norm{u_1}_{H^{1}(\Omega)}
+\norm{F}_{H^{1}(0,T;L^2(\Omega))}\right),
\end{equation}
where the constant $C_0=\norm{U_0}_{\mathcal C([0,T];\mathcal B(\mathcal H))}$ still depends on 
$\Omega$, $T$, $\rho$, $c$, $b_j$ and $a_{ij}$ for $1 \leq i,j\leq n$. 
Now, since the $\mathcal C([0,T];H^2(\Omega) \cap H^1_0(\Omega)) \cap \mathcal C^1([0,T];H_0^1(\Omega))$-solution to
\eqref{sy:intro} is the first component of the $\mathcal C([0,T];\mathcal H)$-solution to 
the integral equation \eqref{FP-eq2}, we are left with the task of showing that some iterate of 
$\mathcal G$ is a contractive map on $\mathcal C([0,T];\mathcal H)$. 
This essentially boils down to the following estimate
\begin{equation} 
\label{es3}
\norm{Q(t)}_{\mathcal B(\mathcal H)} \leq C_2 \frac{t^{-\alpha_2}}{\Gamma(1-\alpha_2)},\ t\in(0,T),
\end{equation}
where $\alpha_2:=\frac{1+\alpha_1}{2}$ and $C_2$ is a positive constant depending only on 
$\Omega$, $T$, $\alpha$ and $q$. 
Indeed, since
$\alpha$ and $q$ lie in $W^{1,\infty}(\Omega)$, then it is clear from \eqref{def-Q} that 
$Q(t) \in \mathcal B(\mathcal H)$ for all $t\in (0,T]$. Moreover, we have
\begin{eqnarray*}
\nabla \left( \frac{q t^{-\alpha}}{\Gamma(1-\alpha)}\right)
&= &  \nabla \left( \frac{q}{\Gamma(1-\alpha)} \right) t^{-\alpha} + \frac{q}{\Gamma(1-\alpha)} \nabla t^{-\alpha} \\
&= & \Gamma(1-\alpha_2) t^{\alpha_2-\alpha} \left( \nabla \left( \frac{q}{\Gamma(1-\alpha)} \right) - \frac{q\nabla \alpha}{\Gamma(1-\alpha)} \ln t \right)   \frac{t^{-\alpha_2}}{\Gamma(1-\alpha_2)}
\end{eqnarray*}
by direct  computation, hence by taking into account that $t \mapsto t^{\alpha_2-\alpha} (1+ |\ln t|)$ 
is uniformly bounded in $(0,T]$, we get that
$$
\norm{\nabla \left( \frac{q t^{-\alpha}}{\Gamma(1-\alpha)} \right)}_{\mathcal B(H^1(\Omega);L^2(\Omega))} 
\leq 
C \frac{t^{-\alpha_2}}{\Gamma(1-\alpha_2)},\ t \in (0,T],
$$
for some positive constant $C$ depending only on $\Omega$, $T$, $\alpha$ and $q$, 
which may change from line to line. Now, \eqref{es3} follows readily from this and the estimate
$\norm{\frac{q t^{-\alpha}}{\Gamma(1-\alpha)}}_{\mathcal{B}(L^2(\Omega))} \leq C_1 \frac{t^{-\alpha_1}}{\Gamma(1-\alpha_1)}$ arising from \eqref{FP-eq3}.

Having established \eqref{es3}, we argue as in the derivation of \eqref{contractive} from 
\eqref{FP-eq3}, where \eqref{es3} 
is substituted for \eqref{FP-eq3}, and find that \eqref{FP-eq2} admits a unique solution 
$U \in \mathcal C([0,T];\mathcal H)$ satisfying
$$
\norm{U}_{\mathcal C([0,T];\mathcal H)}
\leq 
C\left(\norm{u_0}_{H^{2}(\Omega)}+\norm{u_1}_{H^{1}(\Omega)}+\norm{F}_{H^{1}(0,T;L^2(\Omega))}\right).
$$
Here and below, $C$ denotes a positive constant depending only on $T$, $\Omega$, $\rho$, 
$\alpha$, $q$, $c$, $b_j$ and $a_{ij}$ for $1 \leq i,j \leq n$, which may change from line to line. 
As a consequence, there exists a unique solution 
$u\in \mathcal C([0,T]; H^2(\Omega)\cap H^1_0(\Omega))\cap \mathcal C^1([0,T];H^1_0(\Omega))$ 
to \eqref{sy:intro}, such that
\begin{equation}
\label{t1e}
\norm{u}_{\mathcal C([0,T];H^2(\Omega))}+\norm{u}_{\mathcal C^1([0,T];H^1(\Omega))}
\leq 
C\left(\norm{u_0}_{H^{2}(\Omega)}+\norm{u_1}_{H^{1}(\Omega)}+\norm{F}_{H^{1}(0,T;L^2(\Omega))}\right).
\end{equation}
Since $\partial_t^2 u= F -q\partial_t^\alpha u - \mathcal A u\in\mathcal C([0,T];L^2(\Omega))$, 
we get that $u\in \mathcal C^2([0,T];L^2(\Omega))$. 
Moreover, using that $\norm{u}_{\mathcal C^2([0,T];L^2(\Omega))}
=\norm{u}_{\mathcal C^1([0,T];L^2(\Omega))}+\norm{\partial_t^2 u}_{\mathcal C([0,T];L^2(\Omega))}$, 
we obtain
\begin{eqnarray*}
\norm{u}_{\mathcal C^2([0,T];L^2(\Omega))}
&\leq & \norm{u}_{\mathcal C^1([0,T];L^2(\Omega))} + \norm{\mathcal A u}_{\mathcal C([0,T];L^2(\Omega))}+\norm{q}_{L^\infty(\Omega)}\norm{\partial_t^\alpha u}_{\mathcal C([0,T];L^2(\Omega))}+\norm{F}_{\mathcal C([0,T];L^2(\Omega))}\\
&\leq & C \left( \norm{u}_{\mathcal C([0,T];H^2(\Omega))}+\norm{u}_{\mathcal C^1([0,T];L^2(\Omega))}+ \norm{F}_{H^1(0,T;L^2(\Omega))} \right).
\end{eqnarray*}
This and \eqref{t1e} yield \eqref{eq:reg2}, which completes the proof of Theorem \ref{thm:FP}. 

\subsection{Energy estimate}
We turn now to studying the time-evolution of the energy 
\begin{equation}
\label{en0}
E(t) := \int_\Omega \bigg(\frac{1}{\rho(x)}\sum_{i,j=1}^n a_{ij}(x)\partial_i u(x,t)\partial_j u(x,t) 
+ |\partial_t u(x,t)|^2\bigg) dx,\ t\in [0,T],
\end{equation}
of the time-fractional hyperbolic system \eqref{sy:intro} with
initial data $(u_0,u_1) \in H_0^1(\Omega) \times L^2(\Omega)$ and source term $F \in L^2(Q)$. 
Here, $u$ denotes the $\mathcal C([0,T];H_0^1(\Omega)) \cap \mathcal C^1([0,T];L^2(\Omega))$-solution to \eqref{sy:intro}, given by Theorem  \ref{thm:FP}.

The result we have in mind is as follows.

\begin{lem}
\label{lem:EE}
There exists a constant $C>0$, depending on $T$, $\Omega$, $\alpha$, $\rho$, $a_{ij}$, $b_j$ 
for $1 \leq i,j \leq n$, $c$ and $q$, such that the two following estimates hold for all $t \in (0,T]$: 
\begin{equation}
\label{en1}
E(t) \leq C \left( E(0) + \norm{F}_{L^2(Q)}^2 \right)
\end{equation}
\begin{equation}
\label{en2}
E(0) \leq \frac{C}{t} \left( \int_0^t E(s) ds + \norm{F}_{L^2(Q)}^2 \right).
\end{equation}
\end{lem}
\begin{proof}
For all $t \in [0,T]$, we infer from \eqref{ell} and \eqref{en0} that
\begin{equation}
\label{en3}
\min\left\{1, \frac{a_0}{\rho_1}\right\} \left( \norm{\nabla u(\cdot,t)}_{L^2(\Omega)^n}^2 + \norm{\partial_t u(\cdot,t)}_{L^2(\Omega)}^2 \right) \leq E(t) \leq 
\max\left\{1, \frac{a_1}{\rho_0}\right\}  \left( \norm{\nabla u(\cdot,t)}_{L^2(\Omega)^n}^2 + \norm{\partial_t u(\cdot,t)}_{L^2(\Omega)}^2 \right), 
\end{equation}
where we have set $a_1:=\max_{1 \le i \le j \le n} \norm{a_{ij}}_{\mathcal C(\overline{\Omega})}$, hence \eqref{en1} follows immediately from this and \eqref{eq:reg1}. 

We turn now to proving \eqref{en2}. We proceed in two steps.

\noindent {\it Step 1.} First we establish \eqref{en2} for $u \in  \mathcal C^2([0,T];L^2(\Omega)) \cap \mathcal C^1([0,T];H_0^1(\Omega)) \cap  \mathcal C([0,T];H^2(\Omega) \cap H_0^1(\Omega)))$. 
Namely, in accordance with Theorem \ref{thm:FP} we consider a solution $u$ to \eqref{sy:intro} with 
initial states $(u_0,u_1,F) \in (H^2(\Omega) \cap H_0^1(\Omega)) \times H_0^1(\Omega) \times H^1(0,T;L^2(\Omega))$ and we multiply both sides of \eqref{sy:intro} by $\partial_t u$. 
Then, integrating over $\Omega$ we get
$$
\int_\Omega \partial_t u \;\partial_t^2 u\; dx + \int_\Omega q\partial_t u \;\partial_t^\alpha u\; dx 
+ \int_\Omega \partial_t u \bigg( -\frac{1}{\rho} \sum_{i,j=1}^n \partial_i(a_{ij}\partial_j u) + B\cdot\nabla u + cu \bigg) dx
= \int_\Omega \partial_t u \;F\; dx
$$
in $(0,T)$. Bearing in mind that $u=0$ on $\Sigma$, we integrate by parts in 
$\int_\Omega \partial_t u \left( -\frac{1}{\rho} \sum_{i,j=1}^n \partial_i(a_{ij}\partial_j u) \right) dx$ and 
obtain for all $\tau \in (0,T]$ and all $s \in (0,\tau)$ that
$$
-\frac12 \frac{d}{ds}E(s) 
= \bigl< \partial_s u(\cdot,s) , q \partial_s^\alpha u(\cdot,s) +  \sum_{i,j=1}^n a_{ij} \partial_i\Big(\frac{1}{\rho}\Big)\partial_j u(\cdot,s) +B\cdot\nabla u(\cdot,s) + c u(\cdot,s) - F(\cdot,s) \bigr>_{L^2(\Omega)},
$$
where the symbol $\langle \cdot , \cdot \rangle_{L^2(\Omega)}$ denotes the usual scalar product in $L^2(\Omega)$.
Thus, by H\"older's inequality, we have
\begin{eqnarray}
-\frac12 \frac{d}{ds}E(s) 
& \leq & C \left( \norm{\partial_s u(\cdot,s)}_{L^2(\Omega)}^2 +   \norm{\partial_s^\alpha u(\cdot,s)}_{L^2(\Omega)}^2 + \norm{u(\cdot,s)}_{H^1(\Omega)}^2 + \norm{F(\cdot,s)}_{L^2(\Omega)}^2 \right),\ \nonumber \\
& \leq & C \left(E(s)  +  \norm{\partial_s^\alpha u(\cdot,s)}_{L^2(\Omega)}^2 + \norm{F(\cdot,s)}_{L^2(\Omega)}^2\right),\ s \in (0,\tau), \label{en4}
\end{eqnarray}
where $C$ is a positive constant depending on $\Omega$, $\rho$, $a_{ij}$, $b_j$, $c$ and $q$. 
In the last line, we applied 
the Poincar\'e inequality to $u(\cdot,s) \in H_0^1(\Omega)$ and used the left-hand side of \eqref{en3}. 
Further, integrating 
\eqref{en4} over $(0,\tau)$ yields
\begin{equation}
\label{en5}
E(0) \leq E(\tau) + C \left( \int_0^\tau E(s) ds  +  \norm{\partial_t^\alpha u}_{L^2(\Omega \times (0,\tau))}^2 + \norm{F}_{L^2(\Omega \times (0,\tau))}^2\right),\ \tau \in (0,T]. 
\end{equation}
Next, applying Young's convolution inequality, we get that 
$\norm{\partial_t^\alpha u}_{L^2(\Omega \times (0,\tau))} \leq \kappa \norm{\partial_t u}_{L^2(\Omega \times (0,\tau))}$ for some positive constant $\kappa$ depending only on $\alpha$ and $T$, hence \eqref{en5} reads
\begin{equation}
\label{en6}
E(0) \leq E(\tau) + C \left( \int_0^\tau E(s) ds + \norm{F}_{L^2(\Omega \times (0,\tau))}^2  \right),\ \tau \in (0,T],
\end{equation}
where the constant $C$ depends not only on $\Omega$, $\rho$, $a_{ij}$, $b_j$, $c$ and $q$, 
but also on $T$ and $\alpha$. Now, for $t \in (0,T]$ fixed, we integrate \eqref{en6} with respect to 
$\tau$ over $(0,t)$ and find that
$$ t E(0) \leq (1+CT) \int_0^t E(s) ds + CT \norm{F}_{L^2(\Omega \times (0,t))}^2, $$
which yields \eqref{en2} for $(u_0,u_1) \in (H^2(\Omega) \cap H_0^1(\Omega)) \times H_0^1(\Omega)$ and $F \in H^1(0,T;L^2(\Omega))$. 
We turn now to extending this result to the case where 
$(u_0,u_1) \in H_0^1(\Omega) \times L^2(\Omega)$ and $F \in L^2(Q)$.

\noindent {\it Step 2.} Let us consider the 
$\mathcal C([0,T];H_0^1(\Omega)) \cap \mathcal C^1([0,T];L^2(\Omega))$-solution $u$ to 
\eqref{sy:intro} associated with 
$(u_0,u_1) \in H_0^1(\Omega) \times L^2(\Omega)$ and $F \in L^2(Q)$, given by Theorem \ref{thm:FP}. 
By density of $\mathcal C^\infty_{c}(\Omega)$ in $H_0^1(\Omega)$ and in $L^2(\Omega)$, 
and of $\mathcal C^\infty_{c}(Q)$ in $L^2(Q)$, there exist three sequences $\{u_{0,m}\}_{m=1}^{\infty} \in \mathcal C^\infty_{c}(\Omega)^\N$, $\{u_{1,m}\}_{m=1}^{\infty} \in \mathcal C^\infty_{c}(\Omega)^\N$ and $\{F_m\}_{m=1}^{\infty} \subset \mathcal C^\infty_{c}(Q)^\N$ such that
\begin{equation}
\label{en7}
\lim_{m\to+\infty} \|u_0 - u_{0,m}\|_{H_0^1(\Omega)} = 0, \quad
\lim_{m\to+\infty} \|u_1 - u_{1,m}\|_{L^2(\Omega)} = 0, \quad
\lim_{m\to+\infty} \|F - F_m\|_{L^2(Q)} = 0.
\end{equation}
For $m \in \N$, we denote by $u_m$ the $\mathcal C^2([0,T];L^2(\Omega)) \cap \mathcal C^1([0,T];H_0^1(\Omega)) \cap  \mathcal C([0,T];H^2(\Omega) \cap H_0^1(\Omega)))$-solution to the initial-boundary value problem \eqref{sy:intro} where $(u_{0,m}, u_{1,m}, F_m)$ is substituted for $(u_0, u_1, F)$. 
Evidently we have
\begin{equation}
\label{en8}
E_m(0) \leq \frac{C}{t} \left( \int_0^t E_m(s) ds + \norm{F_m}_{L^2(Q)}^2 \right)
\end{equation}
from {\it Step 1}, where $E_m(t)$ denotes the energy obtained by substituting $u_m$ for $u$ in \eqref{en0}. Next, 
since $u-u_m$ is a solution to \eqref{sy:intro} where $(u_0,u_1,F)$ is replaced by $(u_0-u_{0,m},u_1-u_{1,m},F-F_m) \in H_0^1(\Omega) \times L^2(\Omega) \times L^2(Q)$, then we have
\begin{eqnarray}
& & \norm{u-u_m}_{\mathcal C([0,T];H^1_0(\Omega))} + \norm{u-u_m}_{\mathcal C^1([0,T];L^2(\Omega))} \nonumber \\
& \leq &
C\left(\norm{u_0-u_{0,m}}_{H_0^1(\Omega)}+\norm{u_1-u_{1,m}}_{L^2(\Omega)}+\norm{F-F_m}_{L^1(0,T;L^2(\Omega))}\right), \label{en9}
\end{eqnarray}
by virtue of \eqref{eq:reg1}. Further, with reference to \eqref{en0} we get for all $t \in [0,T]$ that
\begin{eqnarray*}
& & E(t)-E_m(t) \\
& = & \int_{\Omega} \left( 
\frac{1}{\rho} \sum_{i,j=1}^n a_{ij} \left( \partial_i(u-u_m) \partial_j u + \partial_i u_m \partial_j (u-u_m) \right) 
- \partial_t(u-u_m) \left( \partial_t(u-u_m) - 2 \partial_t u \right) \right) dx,
\end{eqnarray*}
hence
\begin{eqnarray*}
& & \abs{E(t)-E_m(t)} \\
%& \le &  \frac{M}{\rho_0} \left( \norm{u}_{\mathcal C([0,T];H^1_0(\Omega))} + \norm{u_m}_{\mathcal C([0,T];H^1_0(\Omega))} \right) \norm{u-u_m}_{\mathcal %C([0,T];H^1_0(\Omega))} \\
%& & + \left( 2 \norm{u}_{\mathcal C^1([0,T];L^2(\Omega))} + \norm{u-u_m}_{\mathcal C^1([0,T];L^2(\Omega))} \right) \norm{u-u_m}_{\mathcal %C^1([0,T];L^2(\Omega))}\\
& \le &  \frac{M}{\rho_0} \left( 2 \norm{u}_{\mathcal C([0,T];H^1_0(\Omega))} + \norm{u-u_m}_{\mathcal C([0,T];H^1_0(\Omega))} \right) \norm{u-u_m}_{\mathcal C([0,T];H^1_0(\Omega))} \\
& & + \left( 2 \norm{u}_{\mathcal C^1([0,T];L^2(\Omega))} + \norm{u-u_m}_{\mathcal C^1([0,T];L^2(\Omega))} \right) \norm{u-u_m}_{\mathcal C^1([0,T];L^2(\Omega))}.
\end{eqnarray*}
Finally, putting this together with \eqref{eq:reg1}, \eqref{en7} and \eqref{en9}, we obtain \eqref{en2} by sending $m$ to infinity in \eqref{en8}, which
terminates the proof of Lemma \ref{lem:EE}.
\end{proof}

%%%%%%%%%%%%%%%%%%%%%%%%%%%%%%%%%%%%%%%%%%%%%%%%%
%%%%%%%%%%%%%%%%%%%%%%%%%%%%%%%%%%%%%%%%%%%%%%%%%
\section{Carleman estimate}
\label{sec:CE}

In this section we design a Carleman inequality specifically for the system described by \eqref{sy1}. 
This estimate is the main tool for the analysis of the inverse problems under investigation in this article and we shall derive it from a suitable Carleman inequality for the wave equation, that we establish beforehand in the following subsection. We point out that since there is no zero trace condition at final time $t=T$ in \eqref{sy1}, it is more appropriate for the treatment of the inverse problems discussed in Section \ref{sec:pr-thm:stab} to follow the same path as in the derivation of \cite[Lemma 1]{HIY20} by incorporating
the trace term into the right-hand side of the Carleman inequality.

%In this section we design a Carleman estimate for the system \eqref{sy1} under the assumption that either of the two %initial values $u_0 \in H_0^1(\Omega)$ or $u_1 \in L^2(\Omega)$ vanishes in $\Omega$, i.e. that 
%\begin{equation}
%\label{assump}
%\mbox{(\romannumeral1)}\ \ u_0 = 0\ \mbox{in}\ \Omega\ \hspace*{.5cm} \mbox{or}\ \hspace*{.5cm} 
%\mbox{(\romannumeral2)}\ \ u_1 = 0\ \mbox{in}\ \Omega.
%\end{equation}
%We shall derive an appropriate Carleman inequality for \eqref{sy1} from an existing Carleman estimate for the wave equation. As we are aiming 
%for a Carleman estimate over the interval $(-T,T)$  first step of the method is  applications to inverse problems, we need the estimations of the solution over the time span $(-T,T)$.  

\subsection{A hyperbolic Carleman inequality}
%Weight function and estimation of time-fractional derivative $\partial_t^\alpha v$}

Let $x_0 \in \R^n \setminus \overline{\Omega}$. We consider a sub-boundary $\Gamma_0 \subset \partial\Omega$ such that
\begin{equation}
\label{con:Gamma0}
\{x\in \partial\Omega;\ (x-x_0)\cdot \nu \geq 0\} \subset \Gamma_0
\end{equation}
and we put $\tilde{\Sigma}_0 := \Gamma_0\times (-T,T)$.

Next for $\beta \in (0,1)$ and $\lambda \in (0,+\infty)$ we introduce the following quadratic weight 
function (which is quite usual in the theory of Carleman estimates for the wave equation, 
see e.g., \cite{BY17}),
\begin{equation}
\label{con:varphi}
\varphi(x,t) := e^{\lambda \psi(x,t)}\ \mbox{where}\ \psi(x,t) := \abs{x-x_0}^2 - \beta t^2
\quad (x,t) \in \tilde{Q}:= \Omega \times (-T,T).
\end{equation}
%Here $\beta_0 \in [0,+\infty)$ is chosen in such a way that the function $\varphi$ is positive on $Q$, i.e.
%$\beta_0 > \beta T^2- d_0^2$.
Notice that $\varphi$ depends on the two parameters $\beta$ and $\lambda$, and should rather be 
denoted $\varphi_{\beta,\lambda}$, but for the sake of notational simplicity, 
we omit these dependences. 

Keeping in mind the application to inverse problems in Section \ref{sec:pr-thm:stab}, where we want to avoid the inadequate use of additional initial data $u_j$, $j=0,1$, we aim for a hyperbolic Carleman estimate on the time interval $(-T,T)$, rather than $(0,T)$. 

\begin{lem}
\label{lem:ce-hyper}
%Let $\varphi$ be defined by \eqref{con:varphi}. Assume \eqref{c-T}.
Assume that $B=(b_1,\ldots,b_n)^T \in L^\infty(\Omega)^n$ and that $c \in L^\infty(\Omega)$.Then there exists $\lambda_0 > 0$ such that
for all $\lambda \ge \lambda_0$, there exist two positive constants $s_0=s_0(\lambda)$ and $C=C(\lambda)$ such that for all $s \geq s_0$, we have
\begin{equation}
\label{ce-hyp}
s \norm{e^{s\varphi}\nabla_{x,t} u}_{L^2(\tilde{Q})}^2 + s^3  \norm{e^{s\varphi} u}_{L^2(\tilde{Q})}^2
\leq C \left( \norm{e^{s\varphi} L_0 u}_{L^2(\tilde{Q})}^2 + s \norm{e^{s\varphi} \partial_\nu u}_{L^2(\tilde{\Sigma}_0)} ^2
+ \sum_{\tau=\pm T} \mathcal{E}_{s,\tau}(u) \right) 
\end{equation}
provided $u \in \mathcal C([-T,T];H_0^1(\Omega)) \cap \mathcal C^1([-T,T];L^2(\Omega))$ satisfies 
$L_0 u :=  \partial_t^2 u - \Delta u + B \cdot \nabla u + c u \in L^2(\tilde{Q})$.
% and $\partial_\nu u \in L^2(\tilde{\Sigma}_0)$. 
Here we have set:
$$ \mathcal{E}_{s,\tau}(u) := s \norm{e^{s\varphi(\cdot,\tau)} \nabla_{x,t} u(\cdot,\tau)}_{L^2(\Omega)}^2 + s^3 \norm{e^{s\varphi(\cdot,\tau)} u(\cdot,\tau)}_{L^2(\Omega)}^2,\ \tau = \pm T. $$
\end{lem}
\begin{proof} 
{\it Step 1.} 
Assume that $u\in H^2(-T,T;L^2(\Omega)) \cap L^2(-T,T;H^2(\Omega)\cap H_0^1(\Omega))$.
Interpolating, we have $u \in H^1(-T,T;H_0^1(\Omega))$. %and hence $u \in H^2(\tilde{Q})$. 
Therefore we obtain \eqref{ce-hyp}
directly from \cite[Theorem 4.2]{BY17}. Notice that unlike \cite[Theorem 4.2]{BY17} where it is assumed that $u(\cdot, \pm T) = \partial_t u(\cdot, \pm T) = 0$, there are actual trace terms at $t=\pm T$ on the right-hand side of 
\eqref{ce-hyp}, which are gathered in the sum $\mathcal{E}_{s,-T}(u)+\mathcal{E}_{s,T}(u)$ (these quantities 
arise from the integrations by parts performed in the proof of \cite[Lemmas 4.2 and 4.3]{BY17}).

\noindent {\it Step 2.} Let us now suppose that $u\in \mathcal C^1([-T,T]; L^2(\Omega))\cap \mathcal C([-T,T]; H_0^1(\Omega))$ satisfies $L_0 u\in L^2(\tilde{Q})$. Set
$a := u(\cdot,-T) \in H_0^1(\Omega)$, $b := \partial_t u(\cdot,-T) \in L^2(\Omega)$ and $F := L_0 u \in L^2(\tilde{Q})$. Then, by density of $\mathcal C^\infty_{c}(\Omega)$ in $H_0^1(\Omega)$ and in $L^2(\Omega)$, and of $\mathcal C^\infty_{c}(\tilde{Q})$ in $L^2(\tilde{Q})$, we pick three sequences $\{a_m\}_{m=1}^{\infty} \in \mathcal C^\infty_{c}(\Omega)^\N$, $\{b_m\}_{m=1}^{\infty} \in \mathcal C^\infty_{c}(\Omega)^\N$ and $\{F_m\}_{m=1}^{\infty} \subset \mathcal C^\infty_{c}(\tilde{Q})^\N$ such that
\begin{equation}
\label{prf-eq2}
\lim_{m\to+\infty} \|a - a_m\|_{H_0^1(\Omega)} = 0, \quad
\lim_{m\to+\infty} \|b - b_m\|_{L^2(\Omega)} = 0, \quad
\lim_{m\to+\infty} \|F - F_m\|_{L^2(\tilde{Q})} = 0.
\end{equation}
For $m \in \N$, we denote by $u_m$ the solution to the following initial-boundary value problem
\begin{equation}
\label{phce0}
\left\{
\begin{aligned}
& L_0 u_m = F_m                                                                           &&\quad \mbox{in}\ \tilde{Q}\\
& u_m = 0                                                                                      &&\quad \mbox{on}\ \tilde{\Sigma} := \partial \Omega \times (-T,T)\\
& u_m(\cdot,-T) = a_m,\ \partial_t u_m(\cdot,-T) = b_m                  &&\quad \mbox{in}\  \Omega.
\end{aligned}
\right.
\end{equation}
Then we have $u_m \in \mathcal C([-T,T];H^2(\Omega) \cap H_0^1(\Omega))\cap\mathcal C^1([-T,T];H^1_0(\Omega))\cap \mathcal C^2([-T,T];L^2(\Omega))$ by Theorem \ref{thm:FP} with $q=0$, and consequently 
$$
s \norm{e^{s\varphi} \nabla_{x,t} u_m}_{L^2(\tilde{Q})}^2 + s^3 \norm{e^{s\varphi} u_m}_{L^2(\tilde{Q})}^2 
\leq C \left( \norm{e^{s\varphi} F_m}_{L^2(\tilde{Q})}^2 + s \norm{e^{s \varphi} \partial_\nu u_m}_{L^2(\tilde \Sigma_0)}^2 + \sum_{\tau=\pm T} \mathcal{E}_{s,\tau}(u_m) \right)
$$
for all $s \ge s_0$, from {\it Step 1}. Putting this together with the basic inequality
\begin{eqnarray*}
& & s \norm{e^{s\varphi} \nabla_{x,t} u}_{L^2(\tilde{Q})}^2 + s^3 \norm{e^{s\varphi} u}_{L^2(\tilde{Q})}^2 \\
& \leq & 2 \left( s \norm{e^{s\varphi} \nabla_{x,t} u_m}_{L^2(\tilde{Q})}^2 + s^3 \norm{e^{s\varphi} u_m}_{L^2(\tilde{Q})}^2 + s \norm{e^{s\varphi} \nabla_{x,t} (u-u_m)}_{L^2(\tilde{Q})}^2 + s^3 \norm{e^{s\varphi} (u-u_m)}_{L^2(\tilde{Q})}^2 \right)
\end{eqnarray*}
we get that
\begin{eqnarray*}
& & s \norm{e^{s\varphi} \nabla_{x,t} u}_{L^2(\tilde{Q})}^2 + s^3 \norm{e^{s\varphi} u}_{L^2(\tilde{Q})}^2 \\
& \leq &  C \left( \norm{e^{s\varphi} F_m}_{L^2(\tilde{Q})}^2 + s \norm{e^{s \varphi} \partial_\nu u_m}_{L^2(\tilde \Sigma_0)}^2 + \mathcal{E}_{s}(u_m) + s \norm{e^{s\varphi} \nabla_{x,t} (u-u_m)}_{L^2(\tilde{Q})}^2 + s^3 \norm{e^{s\varphi} (u-u_m)}_{L^2(\tilde{Q})}^2 \right)
\end{eqnarray*}
for all $s \geq s_0$. 
Here and in the remaining part of this proof, $C$ denotes a generic positive constant, 
independent of $s$ and $m$, which may change from line to line, and for the sake of 
notational simplicity we write $\mathcal{E}_{s}$ instead of $\sum_{\tau=\pm T} \mathcal{E}_{s,\tau}$. 
As a consequence we have
\begin{equation}
\label{phce1}
s \norm{e^{s\varphi} \nabla_{x,t} u}_{L^2(\tilde{Q})}^2 + s^3 \norm{e^{s\varphi} u}_{L^2(\tilde{Q})}^2 \leq 
C \left( \norm{e^{s\varphi} F}_{L^2(\tilde{Q})}^2 + s \norm{e^{s \varphi} \partial_\nu u}_{L^2(\tilde \Sigma_0)}^2 +  \mathcal{E}_{s}(u) + \mathcal{R}_s(m) \right)
\end{equation}
for all $s \ge s_0$ and all $m \in \N$, where
\begin{eqnarray*}
\mathcal{R}_s(m) & := & s \norm{e^{s\varphi} \nabla_{x,t} (u-u_m)}_{L^2(\tilde{Q})}^2 + s^3 \norm{e^{s\varphi} (u-u_m)}_{L^2(\tilde{Q})}^2 + s \norm{e^{s \varphi} \partial_\nu (u-u_m)}_{L^2(\tilde \Sigma_0)}^2 \nonumber \\
& & + \mathcal{E}_{s}(u-u_m) + \norm{e^{s\varphi} (F-F_m)}_{L^2(\tilde{Q})}^2. 
\end{eqnarray*}
Thus we are left with the task of proving that $\mathcal{R}_s(m)$ tends to zero as $m$ goes to $+\infty$.
To do that, we use that the weight function $\varphi$ is bounded on the closure of $\tilde{Q}$ by a positive constant independent of $s$, and get that
\begin{eqnarray}
\mathcal{R}_s(m) & \le & e^{sC} \left( s \norm{\nabla_{x,t} (u-u_m)}_{L^2(\tilde{Q})}^2 + s^3 \norm{u-u_m}_{L^2(\tilde{Q})}^2 + s \norm{\partial_\nu (u-u_m)}_{L^2(\tilde \Sigma_0)}^2 \right. \nonumber \\
& & \left. + \sum_{\tau=\pm T} \left( s \norm{\nabla_{x,t}(u-u_m)(\cdot,\tau)}_{L^2(\Omega)}^2 + s^3 \norm{(u-u_m)(\cdot,\tau)}_{L^2(\Omega)}^2 \right) + \norm{F-F_m}_{L^2(\tilde{Q})}^2 \right). \label{prf-eq4}
\end{eqnarray}
Next we have 
\begin{equation*}
\left\{
\begin{aligned}
& L_0 (u - u_m) = F-F_m                                                                   &&\quad \mbox{in}\ \tilde{Q}\\
& u - u_m = 0                                                                                    &&\quad \mbox{on}\ \tilde{\Sigma}\\
& (u-u_m)(\cdot,-T) = a-a_m, \ \partial_t (u-u_m)(\cdot,-T) = b-b_m    &&\quad \mbox{in}\ \Omega,
\end{aligned}
\right.
\end{equation*}
according to \eqref{phce0}, hence
$u-u_m \in \mathcal C([-T,T];H_0^1(\Omega)) \cap \mathcal C^1([-T,T];L^2(\Omega))$ satisfies
\begin{eqnarray}
& & \norm{u-u_m}_{\mathcal C([-T,T];H^1_0(\Omega))} + \norm{\partial_t(u-u_m)}_{\mathcal C([-T,T];L^2(\Omega))} \nonumber \\
& \leq &
C\left(\norm{a-a_m}_{H^1(\Omega)}+\norm{b-b_m}_{L^2(\Omega)}+\norm{F-F_m}_{L^2(\tilde{Q})}\right), \label{prf-eq5}
\end{eqnarray}
by \cite[Lemma 3.2]{BY17}.
Moreover we have $\partial_\nu (u-u_m) \in L^2(\tilde{\Sigma})$ and
\begin{equation}
\label{prf-eq6}
\norm{\partial_\nu (u-u_m)}_{L^2(\tilde{\Sigma})} 
\leq 
C\left(\norm{a-a_m}_{H^1(\Omega)}+\norm{b-b_m}_{L^2(\Omega)}+\norm{F-F_m}_{L^2(\tilde{Q})}\right)
\end{equation}
from \cite[Lemma 3.6]{BY17}. 
Inserting \eqref{prf-eq5}-\eqref{prf-eq6} into \eqref{prf-eq4} then yields
$$
\mathcal{R}_s(m) \le C (1+s+s^3) e^{sC} \left( \norm{a-a_m}_{H_0^1(\Omega)}^2+\norm{b-b_m}_{L^2(\Omega)}^2+\norm{F-F_m}_{L^2(\tilde{Q})}^2 \right),\ n \in \N.
$$
This and \eqref{prf-eq2} entail that $\lim_{m \to +\infty} \mathcal{R}_s(m)=0$ so the desired result follows by sending $m$ to $+\infty$ in \eqref{phce1}.
\end{proof}

\begin{rem} 
\label{rem1}
We have $\partial_\nu u \in L^2(\tilde{\Sigma})$ according to \cite[Theorem 4.1]{L}, whenever $u$ satisfies the conditions of Lemma \ref{lem:ce-hyper}.
\end{rem}

Having established \eqref{ce-hyp}, we are now in position to design a Carleman estimate for \eqref{sy1}.

\subsection{Taking the time-fractional damping term into account}

Our strategy here is to regard the time-fractional damping term as a source term. That is to say that we aim to apply Lemma \ref{lem:ce-hyper} to the first equation of \eqref{sy1} brought into the form $L_0 u = F - q \partial_t^\alpha u$ by moving the time-fractional term $q \partial_t^\alpha u$ to its right-hand side. 
This requires that the solution $u$ to \eqref{sy1} be extended to the time interval $(-T,T)$ and that the time-fractional term $\partial_t^{\alpha} u$ be appropriately estimated.

We shall proceed in three steps: The first step is to estimate $\partial_t^{\alpha} u$ in terms of $\partial_t u$, the second one is to time-symmetrize $u$, and finally the last step is to deduce the desired Carleman inequality from Lemma \ref{lem:ce-hyper} and the two preceding steps.

%We proceed in three steps. First we time-symmetrize the solution to $\eqref{sy1}.
%we move $\partial_t^\alpha u$ to the right-hand side of the master equation of \eqref{sy1} in such a way that
%$$
%L_0 u = \partial_t^2 u - \Delta u + B \cdot \nabla u + c u = F - q\partial_t^{\alpha} u \in L^2(Q).
%$$

\subsubsection{Estimation of the time-fractional damping term}

The estimate we have in mind is as follows.

\begin{lem}
\label{lem:frac}
%Let $\varphi$ be given as \eqref{con:varphi}. 
There exists a constant $C = C(T)>0$ such that for all $u \in H^1(0,T;L^2(\Omega))$, all $\lambda>0$ and all $s \geq 0$, we have
$$
\norm{e^{s\varphi} \partial_t^\alpha u}_{L^2(Q)} \leq C \norm{e^{s\varphi} \partial_t u}_{L^2(Q)}.
$$
\end{lem}
\begin{proof}
For all $x \in \Omega$, $t \mapsto \varphi(x,t)$ is a decreasing function in $(0,T)$, hence
$e^{2s\varphi(x,t)} \leq e^{2 s \varphi(x,\tau)}$ for all $t \in (0,T)$ and all $\tau \in (0,t)$. Thus we have
\begin{eqnarray*}
\norm{e^{s\varphi} \partial_t^\alpha u}_{L^2(Q)}^2 
&= & \int_Q \frac{e^{2s\varphi(x,t)}}{\Gamma(1-\alpha)^2} \abs{\int_0^t (t-\tau)^{-\alpha} \partial_\tau u(x,\tau)d\tau}^2 dxdt 
\nonumber \\
&\leq & \int_Q \frac{1}{\Gamma(1-\alpha)^2} \left( \int_0^t (t-\tau)^{-\alpha} \abs{\partial_\tau u(x,\tau)} e^{s\varphi(x,t)} d\tau \right)^2 dxdt \nonumber \\
&\leq & \int_Q \frac{1}{\Gamma(1-\alpha)^2} \left( \int_0^t (t-\tau)^{-\alpha} \abs{\partial_\tau u(x,\tau)} e^{s\varphi(x,\tau)} d\tau \right)^2 dxdt
\end{eqnarray*}
by direct calculation, and Young's convolution inequality then yields
\begin{eqnarray*}
\norm{e^{s\varphi} \partial_t^\alpha u}_{L^2(Q)}^2&\leq & \int_Q \left(\frac{T^{1-\alpha}}{\Gamma(2-\alpha)}\right)^2 e^{2s\varphi(x,t)} \abs{\partial_t u(x,t)}^2  dxdt \\
& \leq & \left( \frac{\max \{ 1, T \}}{\min_{y \in (1,2)} \Gamma(y)} \right)^2 \norm{e^{s\varphi} \partial_t u}_{L^2(Q)}^2.
\end{eqnarray*}
\end{proof}

\subsubsection{Time-symmetrization}

Let $u$ be the $\mathcal C([0,T];H_0^1(\Omega)) \cap \mathcal C^1([0,T];L^2(\Omega))$-solution to 
\eqref{sy1}, defined by Theorem \ref{thm:FP}. 
Depending on whether the assumption \eqref{assump}\mbox{(\romannumeral1)} or 
\eqref{assump}\mbox{(\romannumeral2)} is fulfilled, 
we introduce the odd or the even extension of $t \mapsto u(\cdot,t)$ on $(-T,T)$, by setting
$$
u(x,t) := 
\left\{
\begin{array}{cl}
u(x,t) & \mbox{if}\ (x,t) \in \Omega \times [0,T) \\
\mp u(x,-t) & \mbox{if}\ (x,t) \in \Omega \times (-T,0).
\end{array} \right.
$$
Here and below the notation $\mp u$ means 
$-u$ in the case of \eqref{assump}\mbox{(\romannumeral1)} and 
$+u$ in the case of \eqref{assump}\mbox{(\romannumeral2)}. 

Then, in light of \eqref{assump} it is easy to check that 
$u \in \mathcal C([-T,T];H_0^1(\Omega)) \cap \mathcal C^1([-T,T];L^2(\Omega))$. 
Moreover, for a.e.
$t \in (-T,0)$ we have
$$
L_0 u(\cdot,t) = \mp L_0 u(\cdot,-t)
= \mp \left( F(\cdot,-t) - q \partial_t^\alpha u(\cdot,-t) \right)\quad \mbox{in}\ \Omega
$$ 
and
$$ 
u(\cdot,t) = \mp u(\cdot,-t) = 0\quad \mbox{on}\ \partial\Omega.
$$ 
Therefore, putting $G(x,t) := F(x,t) - q(x) \partial_t^\alpha u(x,t)$ for a.e. $x \in \Omega$ and all $t \in [0,T)$, and $$
G(x,t) := 
\left\{
\begin{array}{cl}
G(x,t) & \mbox{if}\ (x,t) \in \Omega \times [0,T)\\
\mp G(x,-t) & \mbox{if}\ (x,t) \in \Omega \times (-T,0),
\end{array}
\right.
$$
we end up getting that
\begin{equation}
\label{sy3}
\left\{
\begin{aligned}
& L_0 u = G 
                                                                        &&\quad \mbox{in}\ \tilde{Q}\\
& u = 0                        &&\quad \mbox{on}\ \tilde{\Sigma}\\
& u(\cdot,0) = u_0, \ \partial_t u(\cdot, 0) = u_1    &&\quad \mbox{in}\ \Omega.
\end{aligned}
\right.
\end{equation}
Notice that since it is enough to appropriately extend $t \mapsto G(\cdot,t)$ in $(-T,0)$ for the hyperbolic $L_0 u=G$ to hold in $\tilde{Q}$, we can entirely leave aside the more delicate problem of defining the time-fractional derivative $\partial_t^\alpha u$ for negative time values. 
\begin{rem} 
In a similar fashion to Remark \ref{rem1}, we have $\partial_\nu u \in L^2(\tilde{\Sigma})$ here. This can be seen from Lemma \ref{lem:frac} (with $s=0$) upon arguing as in the proof of \cite[Theorem 4.1]{L}.
\end{rem}

\subsubsection{Completion of the Carleman estimate}

Let us apply Lemma \ref{lem:ce-hyper} to the $\mathcal C([-T,T];H_0^1(\Omega)) \cap \mathcal C^1([-T,T];L^2(\Omega))$-solution $u$ to \eqref{sy3}. 
Taking into account that 
$\norm{e^{s \varphi} G}_{L^2(\tilde Q)}^2=2 \norm{e^{s \varphi} G}_{L^2(Q)}^2$, $\norm{e^{s \varphi} \partial_\nu u}_{L^2(\tilde \Sigma_0)}^2=2 \norm{e^{s \varphi} \partial_\nu u}_{L^2(\Sigma_0)}^2$ 
where we used the notation $\Sigma_0:= \Gamma_0 \times (0,T)$, and that 
$\mathcal{E}_{s,-T}(u)=\mathcal{E}_{s,T}(u)$, we get that
\begin{eqnarray*}
s \norm{e^{s \varphi} \nabla_{x,t} u}_{L^2(\tilde{Q})}^2 + s^3 \norm{e^{s \varphi} u}_{L^2(\tilde{Q})}^2
& \leq & C \left( \norm{e^{s \varphi} G}_{L^2(Q)}^2 + s \norm{e^{s \varphi} \partial_\nu u}_{L^2(\Sigma_0)}^2
+ \mathcal{E}_{s,T}(u) \right) \\
& \leq & C \left( \norm{e^{s \varphi} F}_{L^2(Q)}^2 + \norm{e^{s \varphi} \partial_t^\alpha u}_{L^2(Q)}^2 + s \norm{e^{s \varphi} \partial_\nu u}_{L^2(\Sigma_0)}^2
+ \mathcal{E}_{s,T}(u) \right)
\end{eqnarray*}
for all $\lambda \ge \lambda_0$ and $s \ge s_0$. 
Thus we have
$$
s \norm{e^{s \varphi} \nabla_{x,t} u}_{L^2(Q)}^2 + s^3 \norm{e^{s \varphi} u}_{L^2(Q)}^2 \\
\leq C \left( \norm{e^{s \varphi} F}_{L^2(Q)}^2 + \norm{e^{s \varphi} \partial_t u}_{L^2(Q)}^2 + s \norm{e^{s \varphi} \partial_\nu u}_{L^2(\Sigma_0)}^2
+ \mathcal{E}_{s,T}(u) \right)
$$
by Lemma \ref{lem:frac}, and consequently
$$
s \norm{e^{s \varphi} \nabla_{x,t} u}_{L^2(Q)}^2 + s^3 \norm{e^{s \varphi} u}_{L^2(Q)}^2 \\
\leq C \left( \norm{e^{s \varphi} F}_{L^2(Q)}^2 + s \norm{e^{s \varphi} \partial_\nu u}_{L^2(\Sigma_0)}^2
+ \mathcal{E}_{s,T}(u) \right)
$$
upon taking $s \ge s_1 :=\max(s_0, 2C)$ and possibly enlarging $C$. Therefore, bearing in mind for all $s \geq s_1$ and all $(x,t) \in Q$
that $s e^{2s \varphi(x,t)} \leq e^{Cs}$ for some positive constant $C$ which is independent of $s$, 
we obtain the following result.
\begin{pr}
\label{pr:ce}
Assume \eqref{assump}. 
Then, for all $\lambda \geq \lambda_0$, where $\lambda_0$ is the same as in Lemma \ref{lem:ce-hyper}, there exist two constants $s_1 >0$ and $C > 0$ such that the estimate
\begin{eqnarray*}
& & s \norm{e^{s \varphi} \nabla_{x,t} u}_{L^2(Q)}^2 + s^3 \norm{e^{s \varphi} u}_{L^2(Q)}^2 \\
& \leq & C \left( \norm{e^{s \varphi} F}_{L^2(Q)}^2 + e^{sC} \norm{\partial_\nu u}_{L^2(\Sigma_0)}^2
+ s \norm{e^{s \varphi(\cdot,T)} \nabla_{x,t} u(\cdot,T)}_{L^2(\Omega)}^2 + s^3 \norm{e^{s \varphi(\cdot,T)}  u(\cdot,T)}_{L^2(\Omega)}^2 \right)
\end{eqnarray*}
holds for all $s \geq s_1$. Here, $u$ is the solution to \eqref{sy1} associated with $(u_0,u_1,F) \in H_0^1(\Omega) \times L^2(\Omega) \times L^2(Q)$, defined by Theorem \ref{thm:FP}.
\end{pr}
%\begin{rem}
%As we mentioned in the beginning of this section, Proposition \ref{pr:ce} also holds true for the initial-boundary value %problem \eqref{sy:intro} under some suitable assumptions on $\mathcal{A}$, see e.g., \cite{BY17, H19}. 
%\end{rem}
%\begin{rem}
%\label{rem:beta}
%In the statement, ``sufficiently small $\beta>0$" means that there exists constant $\beta_0>0$ such that for any fixed $0< %\beta < \beta_0$, Proposition \ref{pr:ce} holds true. Moreover, constant $\beta_0 > 0$ depends on the coefficients, but is %independent of $T$ as long as the coefficients of the second-order derivatives are $t$-independent. Indeed, we can take %$\beta_0 = 1$ in the case of \eqref{sy1}. 
%\end{rem}

%\begin{equation}
%\label{def:d}
%d_0 := \min_{x\in \overline{\Omega}} |x-x_0|^2 > 0,\quad 
%d_1 := \max_{x\in \overline{\Omega}} |x-x_0|^2 < \infty. 
%\end{equation}

%Pick $\beta \in (0,1)$ and choose $T$ so large that
%\begin{equation}
%\label{c-T}
%\beta T^2 > 2(d_1-d_0).
%\end{equation}

%%%%%%%%%%%%%%%%%%%%%%%%%%%%%%%%%%%%%%%%%%%%%%%%%
%%%%%%%%%%%%%%%%%%%%%%%%%%%%%%%%%%%%%%%%%%%%%%%%%

\section{Proof of Theorems \ref{thm:stab1} and \ref{thm:stab}}
\label{sec:pr-thm:stab}
%In this section, we recover the initial conditions of the time-fractional wave equation by the key Carleman estimate in %Section \ref{sec:CE} and Lemma \ref{lem:EE} in Section \ref{sec:FP}.

Inspired by \cite{HIY20}, we prove in this section the two main results of this article, stated in Theorems \ref{thm:stab1} and \ref{thm:stab}. 

Prior to doing that we recall that $\Gamma_0 \subset \Gamma$ satisfies the condition \eqref{con:Gamma0} for some $x_0 \in \R^n \setminus \overline{\Omega}$. Next we introduce
\begin{equation}
\label{def-d0&1}
d_0 :=  \min_{x \in \overline{\Omega}} \abs{x-x_0}\ \mbox{and}\ d_1 :=  \max_{x \in \overline{\Omega}} \abs{x-x_0},
\end{equation}
set
$$ T_0 := \sqrt{2(d_1^2-d_0^2)} $$
and we pick $\beta$ so large in $(0,1)$ that the following inequality
\begin{equation}
\label{con:beta}
\beta T^2 \ge 2(d_1^2-d_0^2)
\end{equation}
holds whenever $T>T_0$.

We start by showing Theorem \ref{thm:stab1}.

\subsection{Proof of Theorem \ref{thm:stab1}}

Let us apply Proposition \ref{pr:ce} with $\lambda=\lambda_0$ to the solution $u$ of \eqref{sy1}. We find for all $s \geq s_1$ that 
\begin{eqnarray}
& & s \norm{e^{s \varphi} \nabla_{x,t} u}_{L^2(Q)}^2 \nonumber \\
& \leq & 
C \left( \norm{e^{s \varphi} F}_{L^2(Q)}^2 + e^{C s} \norm{\partial_\nu u}_{L^2(\Sigma_0)}^2 
+ s \norm{e^{s \varphi(\cdot,T)} \nabla_{x,t} u(\cdot,T)}_{L^2(\Omega)}^2 + s^3 \norm{e^{s \varphi(\cdot,T)} u(\cdot,T)}_{L^2(\Omega)}^2 \right) \nonumber \\ 
& \le & C \left( e^{C s} \left( \norm{F}_{L^2(Q)}^2 +  \norm{\partial_\nu u}_{L^2(\Sigma_0)}^2 \right)
+ e^{M_1 s} \left( s \norm{\nabla_{x,t} u(\cdot,T)}_{L^2(\Omega)}^2 + s^3 \norm{u(\cdot,T)}_{L^2(\Omega)}^2 \right) \right).\label{p1.1}
\end{eqnarray}
Here, $C$ denotes a generic positive constant which is independent of $s$ and 
$M_1:=2 e^{\lambda_0 (d_1^2-\beta T^2+\beta_0)}$ where $d_1$ is defined in \eqref{def-d0&1}.
Next, bearing in mind that $u(\cdot,T) \in H_0^1(\Omega)$, we get by combining the Poincar\'e inequality with the energy estimate \eqref{en1} at final time $t=T$,
$\norm{\nabla_{x,t} u(\cdot,T)}_{L^2(\Omega)}^2  \leq C (E(0)+\norm{F}_{L^2(Q)}^2)$, 
that
\begin{equation}
\label{p1.2}
s \norm{\nabla_{x,t} u(\cdot,T)}_{L^2(\Omega)}^2 + s^3 \norm{u(\cdot,T)}_{L^2(\Omega)}^2 \leq C s^3 (E(0)+\norm{F}_{L^2(Q)}^2)
\end{equation}
for all $s \ge 1$. Here we used that $E(T)= \norm{\nabla_{x,t} u(\cdot,T)}_{L^2(\Omega)}^2$ in the framework of Theorem \ref{thm:stab1}. Thus, upon possibly enlarging $C$ in such a way that $s^3 e^{M_1 s} \leq e^{Cs}$ for $s \geq s_1$, we infer from \eqref{p1.1}-\eqref{p1.2} that
\begin{equation}
\label{eq6}
s \norm{e^{s\varphi} \nabla_{x,t} u}_{L^2(Q)}^2
\leq C \left( e^{Cs} \left( \norm{F}_{L^2(Q)}^2 + \norm{\partial_\nu u}_{L^2(\Sigma_0)}^2 \right) +
s^3 e^{M_1 s} E(0) \right),\ s \ge s_2:= \max \{ 1, s_1 \}.
\end{equation} 
On the other hand we have
\begin{equation}
\label{p1.3}
\norm{e^{s\varphi} \nabla_{x,t} u}_{L^2(Q)}^2
\geq \int_0^{\frac{T}{2}} \int_\Omega e^{2 s\varphi(x,t)} \abs{\nabla_{x,t} u(x,t)}^2 dxdt 
\geq e^{M_0 s }\int_0^\frac{T}{2} E(t) dt,
\end{equation}
where $M_0 := 2 e^{\lambda_0 \left( d_0^2 - \beta \frac{T^2}{4} + \beta_0 \right)}$, $d_0$ being the same as in \eqref{def-d0&1}.
Therefore, in light of \eqref{en2}, giving $\int_0^\frac{T}{2} E(t) dt \geq C E(0)-\norm{F}_{L^2(Q)}^2$, we deduce from \eqref{eq6}-\eqref{p1.3} that
\begin{equation}
\label{p1.4}
s e^{s M_0} \left(1-C s^2 e^{-s(M_0-M_1)} \right) E(0) 
\leq C e^{Cs} \left( \norm{F}_{L^2(Q)}^2 + \norm{\partial_\nu u}_{L^2(\Sigma_0 )}^2 \right),\ s \geq s_2.
\end{equation}
Further, since $M_0-M_1=M_0 \left(1 - e^{-\lambda_0 \left(3 \beta \frac{T^2}{4} - (d_1^2-d_0^2)\right)}\right) \geq M_0 \left(1-e^{-\lambda_0 \beta \frac{T^2}{4}} \right)>0$ from \eqref{con:beta}, 
we may choose $s \in [s_2,+\infty)$ so large in \eqref{p1.4} that
$$
E(0) \leq C \left( \norm{F}_{L^2(Q)}^2 + \norm{\partial_\nu u}_{L^2(\Sigma_0)}^2 \right).
$$
This and the Poincar\'e inequality yield Theorem \ref{thm:stab1}.

%\begin{rem}
%\label{rem:beta1}
%We mention that $\beta>0$ exists while $T$ is sufficiently large. By using the constant $\beta_0$ in Remark 
%\ref{rem:beta} and constants $d_0, d_1$ above, we need $T > T_0 = \sqrt{\frac{d_1-d_0}{\beta_0}}$. 
%\end{rem}

%%%%%%%%%%%%%%%%%%%%%%%%%%%%%%%%%%%%%%%%%%%%%%%%%
%%%%%%%%%%%%%%%%%%%%%%%%%%%%%%%%%%%%%%%%%%%%%%%%%
\subsection{Proof of Theorem \ref{thm:stab}}
%\label{sec:pr-thm:stab2}
%In this section, we give the detailed proof of Theorem \ref{thm:stab} by using the key Carleman estimate established in %the former section. The idea is to apply Bukhgeim-Klibanov method. 
%To this end, we first prove the following preliminary estimate. 

We turn now to showing Theorem \ref{thm:stab} by mean of the Carleman estimate of Proposition \ref{pr:ce}. We proceed by applying the Bukhgeim-Klibanov method, see \cite{BK}, which requires the following technical result.

%\subsubsection{Preliminary estimate}
\begin{lem}
\label{lem:pe}
Let $\varphi$ be defined by \eqref{con:varphi}.% for some $\beta \in (0,1)$.
Then there exists a constant $C>0$, depending only on $\Omega$, $T$ and $\varphi$, such that for all $\lambda \geq 0$ and all $s \geq 0$, the following estimate
\begin{equation}
\label{eq2}
\norm{e^{s\varphi(\cdot,0)} \pa_t v(\cdot,0)}_{L^2(\Omega)}^2
\leq 
C \left( \norm{e^{s\varphi} \Box v}_{L^2(Q)}^2 +s \lambda \norm{e^{s\varphi} \nabla_{x,t} v}_{L^2(Q)}^2 +\norm{e^{s\varphi(\cdot,T)} \nabla_{x,t} v(\cdot,T)}_{L^2(\Omega)}^2 \right)
\end{equation}
holds whenever $v \in \mathcal C^1([0,T]; L^2(\Omega))\cap \mathcal C([0,T];H_0^1(\Omega))$ 
satisfies $\Box v := \partial_t^2 v - \Delta v \in L^2(Q)$. 
\end{lem}
\begin{proof}
Let us prove \eqref{eq2} for $v \in \mathcal C^2([0,T];L^2(\Omega)) \cap \mathcal C([0,T];H^2(\Omega) \cap H_0^1(\Omega))$.
In this case we get
\begin{equation}
\label{eq2a}
\norm{e^{s\varphi(\cdot,0)} \partial_t v(\cdot,0)}_{L^2(\Omega)}^2
=  \norm{e^{s\varphi(\cdot,T)} \partial_t v(\cdot,T)}_{L^2(\Omega)}^2 - \int_Q \partial_t \abs{e^{s\varphi} \partial_t v}^2\ dx dt
\end{equation}
and
\begin{eqnarray}
\int_Q \partial_t \abs{e^{s\varphi} \partial_t v}^2\ dx dt 
&= & 2 \int_Q e^{2s\varphi} \left(s\lambda\varphi(\partial_t \psi) \abs{\partial_t v}^2 
+  (\partial_t^2 v)\partial_t v \right) dxdt 
\nonumber \\
&= & 2\int_Q e^{2s\varphi} \left(s\lambda\varphi (\partial_t \psi) \abs{\partial_t v}^2 
+ (\Box v + \Delta v) \partial_t v \right) dxdt \label{eq2b}
\end{eqnarray}
through direct computation. Next, since $v = 0$ on $\Sigma$, we have
$$
\int_Q e^{2s\varphi} (\Delta v) \partial_t v\   dxdt 
= -\int_Q e^{2s\varphi} \left(\nabla v \cdot \nabla (\partial_t v) 
+ 2s\lambda\varphi  (\nabla \psi \cdot\nabla v) \partial_t v \right) dxdt
$$
by integrating by parts over $\Omega$. Thus, in light of the straightforward identity
\begin{eqnarray*}
2 \int_Q e^{2s\varphi} \nabla v \cdot \nabla (\partial_t v) dx dt
& = & \int_Q e^{2s\varphi} \partial_t \abs{\nabla v}^2 dx dt \\
& = & \norm{e^{s\varphi(\cdot,T)} \nabla v(\cdot,T)}_{L^2(\Omega)}^2
- \norm{e^{s\varphi(\cdot,0)} \nabla v(\cdot,0)}_{L^2(\Omega)}^2
- 2 s \lambda \int_Q e^{2s\varphi}  \varphi (\partial_t \psi) \abs{\nabla v}^2dxdt,
\end{eqnarray*}
we find that
\begin{eqnarray*}
& & -2 \int_Q e^{2s\varphi} (\Delta v) \partial_t v\   dxdt \\
&= & \norm{e^{s\varphi(\cdot,T)} \nabla v(\cdot,T)}_{L^2(\Omega)}^2
- \norm{e^{s\varphi(\cdot,0)} \nabla v(\cdot,0)}_{L^2(\Omega)}^2 
- 2 s \lambda \int_Q e^{2s\varphi}  \varphi \left( (\partial_t \psi) \abs{\nabla v}^2 - 2 ( \nabla\psi \cdot\nabla v) \partial_t v \right) dxdt  \\
&\leq &  \norm{e^{s\varphi(\cdot,T)} \nabla v(\cdot,T)}_{L^2(\Omega)}^2
+ 2 s \lambda  \int_Q e^{2s\varphi}  \varphi \left( \abs{\partial_t \psi} \abs{\nabla v}^2 + 2 \abs{\nabla\psi \cdot\nabla v} \abs{\partial_t v} \right) dxdt.
\end{eqnarray*}
Putting the above estimate together with \eqref{eq2a}-\eqref{eq2b} and using that the two functions 
$\psi$ and $\varphi$ are in $W^{1,\infty}(Q)$, we obtain \eqref{eq2}. 

Now, the claim of Lemma \ref{lem:pe} follows readily from this and the density of 
$\mathcal C^2([0,T];L^2(\Omega)) \cap \mathcal C([0,T];H^2(\Omega) \cap H_0^1(\Omega))$ 
in the space $\{ v \in \mathcal C^1([0,T]; L^2(\Omega))\cap \mathcal C([0,T];H_0^1(\Omega));\ \Box v \in L^2(Q) \}$ endowed with the norm
$\norm{v}_{\mathcal C^1([0,T];L^2(\Omega))} + \norm{v}_{\mathcal C([0,T];H^1(\Omega))} 
+ \norm{\Box v}_{L^2(Q)}$.
\end{proof}

%\subsubsection{}

Armed with Lemma \ref{lem:pe} we are now in position to build the proof of Theorem \ref{thm:stab}.

The first step of the Bukhgeim-Klibanov method is to differentiate once with respect to $t$ on both sides of \eqref{sy1}.
Setting $v := \partial_t u$, we get that
\begin{equation}
\label{sy4}
\left\{
\begin{aligned}
& \partial_t^2 v + q\partial_t^{\alpha} v - \Delta v + B \cdot \nabla v + cv = (\partial_t R)f 
                                                                                                                        &&\quad \mbox{in } Q\\
& v = 0                                                                                                   &&\quad \mbox{on } \Sigma\\
& v(\cdot,0) = \partial_t u(\cdot, 0) = 0, \quad \partial_t v(\cdot,0) = \partial_t^2 u(\cdot,0) = R(\cdot,0)f
                                                                                                              &&\quad \mbox{in } \Omega.
\end{aligned}
\right.
\end{equation}
Here we used the fact that $q$ is independent of $t$ and that
$u(\cdot,0) = \partial_t u(\cdot,0) = 0$ in $\Omega$, in order to write
$$
\pa_t (q \pa_t^\alpha u)(t) = q \pa_t ( \pa_t^\alpha u)(t) = q \pa_t^\alpha (\pa_t u)(t) = q \pa_t^\alpha v(t),\ t\in (0,T),
$$
see e.g., \cite[Section 2.2.5]{P99}. Next, we apply Proposition \ref{pr:ce} with $\lambda=\lambda_0$ to \eqref{sy4}. We obtain that for all $s \geq s_1$,
\begin{eqnarray}
& &s \norm{e^{s \varphi} \nabla_{x,t} v}_{L^2(Q)}^2 +  s^3 \norm{e^{s \varphi}v}_{L^2(Q)}^2 \nonumber \\
& \leq & C \left( \norm{e^{s \varphi}f}_{L^2(Q)}^2 + e^{Cs} \norm{\partial_\nu v}_{L^2(\Sigma_0)}^2 
+ s \norm{e^{s\varphi(\cdot,T)} \nabla_{x,t} v(\cdot,T)}_{L^2(\Omega)}^2 + s^3 \norm{e^{s\varphi(\cdot,T)} v(\cdot,T)}_{L^2(\Omega)}^2 \right), \label{eq3}
\end{eqnarray}
where $C$ denotes a generic positive constant that is independent of $s$.

In view of retrieving the unknown function $f$ from the (second) initial condition of \eqref{sy4},
\begin{equation}
\label{eq4}
\partial_t v(\cdot,0) = R(\cdot,0)f\ \mbox{in}\ \Omega,
\end{equation}
we apply Lemma \ref{lem:pe} to $v$. We get for all $s \geq 1$ that
\begin{eqnarray*}
& & \norm{e^{s\varphi(\cdot ,0)} \partial_t v(\cdot,0)}_{L^2(\Omega)}^2 \\
%&\leq & C \left( \norm{e^{s\varphi} F}_{L^2(Q)}^2 + s \norm{e^{s\varphi} \nabla_{x,t} v}_{L^2(Q)}^2 + 
%\norm{e^{s\varphi(\cdot,T)} \nabla_{x,t} v(\cdot,T)}_{L^2(\Omega)}^2 \right)\\
& \leq & C \left( \norm{e^{s\varphi} ((\partial_t R)f - q\partial_t^\alpha v - B \cdot \nabla v -cv)}_{L^2(Q)}^2 + s \norm{e^{s\varphi} \nabla_{x,t} v}_{L^2(Q)}^2 + 
\norm{e^{s\varphi(\cdot,T)} \nabla_{x,t} v(\cdot,T)}_{L^2(\Omega)}^2 \right) 
\\
&\leq & C \left( \norm{e^{s\varphi} f}_{L^2(Q)}^2 + s \norm{e^{s\varphi} \nabla_{x,t} v}_{L^2(Q)}^2 + s^3 \norm{e^{s\varphi} v}_{L^2(Q)}^2
+ \norm{e^{s\varphi(\cdot,T)} \nabla_{x,t} v(\cdot,T)}_{L^2(\Omega)}^2 \right),
\end{eqnarray*}
%and consequently
%$$ \norm{e^{s\varphi(\cdot ,0)} \partial_t v(\cdot,0)}_{L^2(\Omega)}^2 \leq
%C \left( \norm{e^{s\varphi} f}_{L^2(Q)}^2 + s \norm{e^{s\varphi} \nabla_{x,t} v}_{L^2(Q)}^2 + s^3 \norm{e^{s\varphi} v}_{L^2(Q)}^2
%\norm{e^{s\varphi(\cdot,T)} \nabla_{x,t} v(\cdot,T)}_{L^2(\Omega)}^2 \right) $$
where we used Lemma \ref{lem:frac} in the last line. Thus, with reference to the positivity condition \eqref{con:R}, it follows from \eqref{eq3}-\eqref{eq4} that
\begin{equation}
\label{eq4.0}
\norm{e^{s\varphi(\cdot,0)} f}_{L^2(\Omega)}^2 
\leq C \left( \norm{e^{s\varphi} f}_{L^2(Q)}^2 + e^{Cs} \norm{\partial_\nu v}_{L^2(\Sigma_0)}^2
+ s \norm{e^{s\varphi(\cdot,T)} \nabla_{x,t} v(\cdot,T)}_{L^2(\Omega)}^2 + s^3 \norm{e^{s\varphi(\cdot,T)} v(\cdot,T)}_{L^2(\Omega)}^2 \right)
\end{equation}
for all $s \geq s_2 = \max \{ 1, s_1 \}$. 

Further, since $t \mapsto \varphi(x,t)$ attains its maximum at $t=0$ for all $x \in \Omega$, then
$\norm{e^{s\varphi} f}_{L^2(Q)}$ can be made arbitrarily small relative to $\norm{e^{s\varphi(\cdot,0)} f}_{L^2(\Omega)}$ by taking $s$ sufficiently large. This can be seen by following the same path as in the derivation of \cite[Lemma 5.3]{BY17}, that is to say by writing
\begin{equation}
\label{eq4.1}
\norm{e^{s\varphi} f}_{L^2(Q)}^2 
=\int_{\Omega} e^{2 s \varphi(x,0)} \abs{f(x)}^2 \left( \int_0^T e^{-2s (\varphi(x,0)-\varphi(x,t))} dt \right) dx,\ s \geq 0,
\end{equation}
in accordance with Lebesgue's theorem, and by noticing for all $(x,t) \in Q$ that
\begin{eqnarray}
\varphi(x,0)-\varphi(x,t) & = & e^{\lambda_0 \psi(x,0)} \left(1-e^{\lambda_0 (\psi(x,t)-\psi(x,0))} \right) \nonumber \\
& = & e^{\lambda_0 (\abs{x-x_0}^2+\beta_0)} \left(1-e^{-\lambda_0 \beta t^2} \right) \nonumber \\
& \geq & e^{\lambda_0 (d_0^2+\beta_0)} \left(1-e^{-\lambda_0 \beta t^2} \right) := \zeta(t), \label{eq4.2}
\end{eqnarray}
where $d_0$ is defined in \eqref{def-d0&1}. Indeed, putting \eqref{eq4.1} together with \eqref{eq4.2}, we obtain that
$$ \norm{e^{s\varphi} f}_{L^2(Q)}^2 \leq \kappa(s) \norm{e^{s\varphi(\cdot,0)} f}_{L^2(\Omega)}^2,\ s \geq 0, $$
where the non-negative function $\kappa(s):=\int_0^T e^{-2 s \zeta(t)} dt$ converges to zero 
as $s$ tends to $+\infty$.
Therefore, taking $s$ so large that $\kappa(s) \leq (2C)^{-1}$, we get that $2(\norm{e^{s\varphi(\cdot,0)} f}_{L^2(\Omega)}^2 - C\norm{e^{s\varphi} f}_{L^2(Q)}^2) \geq \norm{e^{s\varphi(\cdot,0)} f}_{L^2(\Omega)}^2$ according to the above estimate, 
which establishes that the first term on the right-hand side of \eqref{eq4.0} can be absorbed into 
its left-hand side. 
As a consequence, there exists
$s_3 \geq s_2$ such that
\begin{equation}
\label{eq4.3}
\norm{e^{s\varphi(\cdot,0)} f}_{L^2(\Omega)}^2 
\leq C \left( e^{Cs} \norm{\partial_\nu v}_{L^2(\Sigma_0)}^2
+ s \norm{e^{s\varphi(\cdot,T)} \nabla_{x,t} v(\cdot,T)}_{L^2(\Omega)}^2 + s^3 \norm{e^{s\varphi(\cdot,T)} v(\cdot,T)}_{L^2(\Omega)}^2 \right),\ s \geq s_3.
\end{equation}

We are thus left with the task of showing that the two last terms on the right-hand side of 
\eqref{eq4.3} are dominated by $\norm{e^{s\varphi(\cdot,0)} f}_{L^2(\Omega)}^2$, 
provided $s$ is large enough. 
This can be done by combining the Poincar\'e inequality, giving
$$ \norm{e^{s\varphi(\cdot,T)} \nabla_{x,t} v(\cdot,T)}_{L^2(\Omega)}^2 + \norm{e^{s\varphi(\cdot,T)} v(\cdot,T)}_{L^2(\Omega)}^2
\leq C e^{2se^{\lambda_0 \left(d_1^2-\beta T^2 + \beta_0 \right)}} \norm{\nabla_{x,t} v(\cdot,T)}_{L^2(\Omega)}^2, $$
with the energy estimate \eqref{en1} at $t=T$,
$$
\norm{\nabla_{x,t} v(\cdot,T)}_{L^2(\Omega)}^2
\leq C \left( \norm{R(\cdot,0)f}_{L^2(\Omega)}^2 +  \norm{(\partial_t R)f}_{L^2(Q)}^2 \right) 
\leq C \norm{f}_{L^2(\Omega)}^2
$$
and
$$ 
\norm{e^{s\varphi(\cdot,0)} f}_{L^2(\Omega)}^2 
\geq e^{2s e^{\lambda_0( d_0^2 + \beta_0 )}} \norm{f}_{L^2(\Omega)}^2,\ s \geq 0,
$$
where we recall that $d_0$ and $d_1$ are defined in \eqref{def-d0&1}. This entails that
$$ \norm{e^{s\varphi(\cdot,T)} \nabla_{x,t} v(\cdot,T)}_{L^2(\Omega)}^2 + \norm{e^{s\varphi(\cdot,T)} v(\cdot,T)}_{L^2(\Omega)}^2
\leq C e^{-2 s e^{\lambda_0(d_0^2+\beta_0)} \delta_0} \norm{e^{s\varphi(\cdot,0)} f}_{L^2(\Omega)}^2,\  s \geq 0, $$ 
where $\delta_0 :=1-e^{-\lambda_0 \left( \beta T^2 -(d_1^2-d_0^2) \right)}$, and consequently that
$$ \left(1-Cs^3 e^{-2 s e^{\lambda_0(d_0^2+\beta_0)} \delta_0} \right)  \norm{e^{s\varphi(\cdot,0)} f}_{L^2(\Omega)}^2 \leq C e^{Cs} \norm{\partial_\nu v}_{L^2(\Sigma_0)}^2,\ s \geq s_3, $$
by \eqref{eq4.3}. Finally, since $\delta_0 \ge 1-e^{-\lambda_0 \beta \frac{T^2}{2}}>0$ from \eqref{con:beta}, the desired result follows upon taking $s$ sufficiently large in the above estimate.

\section*{Acknowledgments}
The first author was supported by Grant-in-Aid for Scientific Research (S) 
15H05740, Grant-in-Aid for Research Activity Start-up 19K23400 
of JSPS (Japan Society for the Promotion of Science) and 
Grant-in-Aid for JSPS Fellows 20F20319.
The second and third authors were partially supported by the Agence 
Nationale de la Recherche under grant ANR-17-CE40-0029. 
The fourth author was supported by Grant-in-Aid for
Scientific Research (A) 20H00117, JSPS and by NSFC (Nos.\! 11771270,
91730303), and this work was done by RUDN University 
Strategic Academic Leadership Program.
%This work was supported by A3 Foresight Program ``Modeling and Computation of 
%Applied Inverse Problems" of Japan Society for the Promotion of Science.

%%%%%%%%%%%%%%%%%%%%%%%%%%%%%%%%%%%%%%%%%%%%%%%%%
%%%%%%%%%%%%%%%%%%%%%%%%%%%%%%%%%%%%%%%%%%%%%%%%%
%%%%%%%%%%%%%%%%%%%%%%%%%%%%%%%%%%%%%%%%%%%%%%%%%
%%%%%%%%%%%%%%%%%%%%%%%%%%%%%%%%%%%%%%%%%%%%%%%%%

\end{document}